\documentclass[12pt]{article}
\usepackage{amsxtra, amsbsy,amsgen,amsopn,  amsfonts, amssymb, amsthm}
\usepackage{mathrsfs}
\usepackage{lscape}
\usepackage{ascmac}
    \newtheorem{thm}{Theorem}[section]%
    \newtheorem{prop}[thm]{Proposition}%
    \newtheorem{cor}[thm]{Corollary}%
    \newtheorem{lem}[thm]{Lemma}%
    \newtheorem{defn}[thm]{Definition}%
    \newtheorem{rem}[thm]{Remark}%

    \numberwithin{equation}{section}
\begin{document}
    \begin{center}
		   {\bf Construction of solutions of nonlinear irregular singular differential equations by 
		   		Borel summable functions and an application to Painlev\'{e} equations }\vspace{5mm} \\
	  		{Sunao {\sc \={O}uchi}}
			\footnote{Sophia Univ. Tokyo Japan, e-mail s\_ouchi@sophia.ac.jp} 
		\end{center} 
	\begin{abstract}
			A system of nonlinear differential equations 
			$x^{1+\gamma}\frac{dY}{dx}= F_0(x)+A(x)Y+F(x,Y)$ is considered. 
			The origin $x=0$ is irregular singular. There exist pioneering works about them.
			We study more precisely than preceding works, 
			the meaning of asymptotic expansion of transformations and solutions by using Borel 
			summable functions in asymptotic analysis, and apply results to Painlev\'{e} equations. 		 
			\footnote {\par 
			\quad Painlev\'{e} equation	\par 
			{2020 Mathematical Classification Numbers. 34M30; Secondary  34M04, 34M55}} 
	\end{abstract}
	\section{\normalsize Introduction}
			The main purpose of the present paper is to construct solutions of the following system 
			of nonlinear differential equations with irregular singularity at $x=0$,  
		\begin{equation}\left \{
				\begin{aligned}
				Y=&{}^t(y_1.y_2,\cdots,y_n) \\ 	
				x^{1+\gamma}\frac{dY}{dx}=& F_0(x)+A(x)Y+F(x,Y)
				\end{aligned} \right . \label{Eq-00}
		\end{equation}
 					There are pioneering researches about \eqref{Eq-00} by Hukuhara \cite{Huk-1}, Malmquist 
 				\cite{Malm-1}, Trjitzinsky \cite{Trj} Iwano \cite{Iwano-1} \cite{Iwano-2} and 
				many other important ones. 
    	  They constructed formal solutions and showed the existence 
				of genuine solutions under some conditions. The theory of multisummable 
				functions in asymptotic analysis has been  developed after their studies (see Balser \cite{Bal}). 
				Borel summability is a special case of multisummability. It was not used in the pioneering 
				researches. For nonlinear equations it is shown  that formal power series 
				solutions of ordinary differential equations are multisummable in Braaksma \cite{Braak} and 
				those of some class of partial differential equations	are multisummable in \={O}uchi \cite{Ouchi}. 
					Let us return to a classical important result due to Malmquist \cite{Malm-1}.
		    Let $\Lambda =\{\lambda_1,\cdots, \lambda_{n}\}$ be the set of eigenvalues of $A(0)$
		    and assume they are distinct. Let $\Lambda' =\{\lambda_1,\cdots, \lambda_{n'}\}$ and 
			  $\Lambda'' =\{\lambda_{n'+1},\cdots, \lambda_{n}\}$. $\Lambda'$ and $\Lambda''$ are separated by a 
			  straight line through the origin in the complex plane. It is shown in \cite{Malm-1} that 
		    there exists an $n'$-parameter family of solutions in some sector corresponding to 
		    $\Lambda'$. It is the main purpose that we try to have another look at this result, by
		    applying theory of Borel summable functions in asymptotic analysis. 
		    We construct transformations and solutions more precisely and clearly in a function class with 
		    some Gevrey type estimates. 
		    For $\gamma=1$  (rank 1) equation \eqref{Eq-00} was treted in Costin,O \cite{Cos-1}, 
		    Costin,O, Costin,R.D \cite{Cos-2} and Braaksma,Kuik \cite{Bra-Kuik} in a different way from 
		    that in the present paper. The structure of solutions were studied there,  
		    by applying the resurgence theory due to \'{E}calle. We use only the elementary properties of 
		    Borel summable functions in this paper. 
		    \par
		    In secton 1 we sum up shortly what we need about Borel summable functions. 
			  In section 2 we study case $F_0(x)\equiv 0$ and give one of main results (Theorem \ref{main-th})
	      that is construction of exponential series solutions called often {\it transseries} solutions. 
		     In section 3 we give the proof of Theorem \ref{main-th} and remarks.
		     In section 4 we study case $F_0(x)\not \equiv 0$, by reducing to the former case. 
		     In section 5 we apply the results to Painlev\'{e} II, IV equations as examples. In section 6
		     we give the proof of Proposition \ref{prop-1} concerning diagonalization of linear systems.    
         The main results in this paper are obtained under the condition that 
         the eigenvalues of matrix $A(0)$ are distinct. 
         It will be studied in another paper for the case multiple eigenvalues of matrix $A(0)$ 
         appear. 
\section{\normalsize Borel summable functions with holomorphic parameters}
			We introduce some notations and definitions.
			Let $I=(\alpha,\beta)$ be an open interval and $\widetilde{\mathbb{C}}_{\{0\}}$ be 
			the universal covering
			space of $\mathbb{C}-\{0\}$. 
			$S(I)=S(\alpha,\beta)=\{x\in \widetilde{\mathbb{C}}_{\{0\}}; \arg x\in I\}$. 
			$S_{0}(I)=S_0(\alpha,\beta)=\{x\in S(I); 0<|x|<\rho(\arg x)\}$,
			where $\rho(t)$ is some positive continuous function on $I$. The same notation 
			$S_{0}(\cdot)$ is	used for various $\rho(\cdot)$. 
			 For arbitrary small $\epsilon>0$, $I_{\epsilon}=
			(\alpha+\epsilon, \beta-\epsilon)\subset I$.  
	 		${\mathscr O}(U)$ is the set of holomorphic functions on a domain $U$. 
	 		${\mathbb C}[[x]]$ is the set of formal power series of $x$. ${\mathbb N}$ is the 
	 		set of nonnegative integers and ${\mathbb Z}$ is the 
	 		set of integers. Let $k=(k_1,\cdots,k_n)\in  {\mathbb N}^n$ and 
	 		$Y=(y_1.\cdots,y_n)\in {\mathbb C}^n$. Then we use notations $k!=k_1!k_2!\cdots k_n!$,
	 		$|k|=\sum_{i=1}^{n}k_i$, 
	 		$Y^{k}=y_1^{k_1}\cdots y_n^{k_n}$, $|Y|=\max_{1\leq i\leq n}|y_i|$ 	
	 		and	$(\frac{\partial}{\partial Y})^k=\prod_{i=1}^{n}(\frac{\partial}{\partial y_i})^{k_i}$.	
	\begin{defn} Let $\kappa>0$, $I=(\alpha,\beta)$ with $\beta-\alpha>\pi/\kappa$ and 
		$\Omega=\{Y\in {\mathbb C}^n;	|Y|<R \}$. A function 
		 $f(x,Y)\in {\mathscr O}(S_{0}(I)\times \Omega)$ is said to be
		$\kappa$-Borel summable with respect to $x$, 
		if there exist constants $M$, $C$ and
		$\{a_n(y)\}_{n=0}^{\infty}\subset {\mathscr O}(\Omega)$ such that
		for any $N\geq 0$
	\begin{equation}
			|f(x,Y)-\sum_{n=0}^{N-1}a_n(Y)x^n|\leq 
			MC^N|x|^N{\Gamma(\frac{N}{\kappa}+1)}\quad 
			(x,Y)\in S_{0}(I)\times \Omega
	\end{equation}
			holds. 
			The totality of $\kappa$-Borel summable functions with respect to $x$ on 
			$S_{0}(I)\times \Omega$ is denoted by 
			$\mathscr O_{\{1/\kappa\}}(S_{0}(I)\times \Omega)$. \par					
			We say that $f(x,Y)$ is $\kappa$-Borel summable  
			in a direction $\theta$, if there exists $\delta>\pi/(2\kappa)$ such that
			$f(x,Y) \in \mathscr O_{\{1/\kappa\}}
			(S_{0}(\theta-\delta,\theta+\delta)\times \Omega)$.  		
	\end{defn} 
	  The notion of Borel summability is originally used for formal power series.  
		\begin{defn}
	  Let $\widetilde{f}(x)=\sum_{n=0}^{\infty}a_n x^n \in \mathbb{C}[[x]]$. $\widetilde{f}(x)$ is said to be
		$\kappa$-Borel summable in a direction $\theta$, if there exists 
		$f(x)\in {\mathscr O}_{\{1/\kappa\}}(S_{0}(I))$,
		 $I=(\theta-\delta,\theta+\delta)\;\; \delta>\pi/2\kappa$, such that
		 for any $N\geq 0$
	\begin{equation}
			|f(x)-\sum_{n=0}^{N-1}a_n x^n|\leq 
			MC^N|x|^N{\Gamma(\frac{N}{\kappa}+1)}\quad x \in S_{0}(I)
	\end{equation}
	holds.  If $f(x)$ exists, then it is unique. Hence $\widetilde{f}(x)$ and $f(x)$ are often identified.
	\end{defn}
	     
	      Let $I=(\alpha,\beta)$ with $\beta-\alpha>\pi/\kappa$, $\theta=(\alpha+\beta)/2$
	      and $\delta=(\beta-\alpha)/2$. Then $I=(\theta-\delta,\theta+\delta)$ and $\delta>\pi/(2\kappa)$.
	      Let  $\psi(x,Y)\in \mathscr{O}(S_0(I)\times \Omega)$ and $|\psi(x,Y)|\leq C|x|^c (c>0)$.  
				$\kappa$-Borel transform of $\psi(x,Y)$  is defined by
			\begin{equation}
				\begin{aligned}
					(\frak{B}_{\kappa}\psi)(\xi,Y)= \frac{1}{2\pi i} \int_{\mathcal C}\exp(\frac{\xi}{x})^{\kappa} 
					\psi(x,Y)dx^{-\kappa},
				\end{aligned}%
			\end{equation}
				where ${\mathcal C}$ is a contour in $S_0(I)$ that starts from $0e^{i(\theta+\delta')}$ to 
				$r_0e^{i(\theta+\delta')}$ on a segment and next on an arc $|t|=r_0$ to  $r_0e^{i(\theta-\delta'')}$ 
				and finally on a segment ends at 
				$0e^{i(\theta-\delta'')}$\; ($\delta>\delta',\delta''>\pi/(2\kappa)$). \par 		
				We denote $(\frak{B}_{\kappa}\psi)(\xi,Y)$ by $\widehat{\psi}(\xi,Y)$.				
				Let $\widehat{\alpha}(\kappa)=\alpha+\pi/2\kappa$,  
				$\widehat{\beta}(\kappa)=\beta-\pi/2\kappa$ and 
				$\widehat{I}({k})=(\widehat{\alpha}(k),\widehat{\beta}(k))$.
				Then $\widehat{\psi}(\xi,Y)$ is holomorphic in an infinite sector $S(\widehat{I}(k))$ 
				with respect $\xi$. 
						We can construct $\psi(x,Y)$ by the $\kappa$-Laplace transform 
				${\frak L}_{\kappa}\widehat{\psi} $, 
				that is, $\psi(x,Y)=({\frak L}_{\kappa}\widehat{\psi})(x,Y)$
		\begin{equation}
			\begin{aligned}\ &
				 ({\mathcal L}_{\kappa}\widehat{\psi})(x,Y)=
				 \int_{0}^{\infty e^{\sqrt{-1}\theta}}e^{-(\frac{\xi}{x})^{k}}
			     \widehat{\psi}(\xi,Y)d\xi^{\kappa} \quad \theta\in {\widehat{I}(\kappa)}. \label{L}
			\end{aligned}
		\end{equation}
				If $f(x,Y)\in {\mathscr O}_{\{1/k\}}(S_{0}(I)\times \Omega)$ with $a_0(Y)=0$, then it 
				is known that there exists $r>0$ such that 
		\begin{equation}
			\widehat{f}(\xi,Y)=\sum_{n=1}^{\infty}\frac{a_n(Y)}{\Gamma(n/\kappa)}\xi^{n-\kappa} \label{B-tr}
		\end{equation}
			holds in $\{0<|\xi|<r\}\times \Omega$. $\xi^{\kappa-1}\widehat{f}(\xi,Y)$ is holomorphic 
			in $\{|\xi|<r\}\times \Omega$. 
		 It holds that for any small $\epsilon>0$  		
		\begin{equation}
			\begin{aligned}
			\ &|\widehat{f}(\xi,Y)|\leq  \frac{K_{\epsilon}|\xi|^{1-\kappa}
			e^{c_{\epsilon}|\xi|^{\kappa}}}{\Gamma(1/\kappa)}
				\quad (\xi,Y)\in \big(\{0<|\xi|<r\}\cup 
				S_{\widehat{I}_{\epsilon}(\kappa)}\big)\times \Omega.  
			\end{aligned} \label{B-est}  
		\end{equation}
			As for the details of Borel summable functions, Borel transform and 
			Laplace transform we refer to Balser \cite{Bal}. 
			By expanding $\widehat{f}(\xi,Y)$ at $Y=0 \in \mathbb C^{n}$,
		\begin{equation}
			\begin{aligned}
			\widehat{f}(\xi,Y)=\sum_{k\in \mathbb N^{n}}\widehat{f}_k(\xi)Y^k,\quad 
			\widehat{f}_k(\xi)=\frac{1}{k!}(\frac{\partial}{\partial Y})^k\widehat{f}(\xi,0)
			\end{aligned}
		\end{equation}
			and Cauchy's inequality, the following holds.
		\begin{lem}	
			Let $f(x,Y)\in \mathscr O_{\{1/\kappa\}}(S_{0}(I)\times \Omega)$ with $f(0,Y)=0$.
			Then 			
				\begin{equation}
		|\widehat{f}_k(\xi)|\leq 
			 \frac{K_{\epsilon}|\xi|^{1-\kappa} e^{c_{\epsilon}|\xi|^{\kappa}}}{R^{|k|}\Gamma(1/\kappa)}\quad  
			 \xi \in \{0<|\xi|<r\}\cup S_{\widehat{I}_{\epsilon}(\kappa)}. 
		\end{equation}
				 \label{lem-Cau}
		\end{lem}
		\begin{prop}	
				Let $\{f_k(x)\}_{k\in {\mathbb N}^n} \subset \mathscr O_{\{1/\kappa\}}(S_{0}(I))$ 
				with $f_k(0)=0$. Suppose that their $\kappa$- 
				Borel transforms $\{\widehat{f}_k(\xi)\}_{k\in {\mathbb N}^n}$ have bounds for any small 
				$\epsilon>0$			
			\begin{equation}
				\begin{aligned}
				|\widehat{f}_k(\xi)|\leq 
			 \frac{K_{\epsilon}|\xi|^{1-\kappa} e^{c_{\epsilon}|\xi|^{\kappa}}}{R^{|k|}\Gamma(1/\kappa)}\quad  
			 \xi \in \{0<|\xi|<r\}\cup S_{\widehat{I}_{\epsilon}(\kappa)}.
				\end{aligned}
			\end{equation}
			 Then
				$f(x,Y)=\sum_{k\in \mathbb N^{n}}f_k(x)Y^k \in \mathscr O_{\{1/\kappa\}}(S_{0}(I_{\epsilon})\times 
				\Omega_0)$, $\Omega_0=\{Y\in {\mathbb C}^n;|Y|<R_0 \}$ $(R_0<R)$.  \label{prop-Borel}
		\end{prop}
		\begin{proof}
				Since ${g}^*(\xi,Y):=\xi^{\kappa-1}\sum_{k\in {\mathbb N}^{n}}\widehat{f}_k(\xi)Y^k$
				 converges in $(\{|\xi|<r\}\cup S_{\widehat{I}_{\epsilon}(\kappa)})\times \Omega_0$, 
				$|{g}^*(\xi,Y)|\leq M_{\epsilon}e^{c_{\epsilon}|\xi|^{\kappa}}$ and there exist 
				$\{a_n(Y)\}_{n\geq 1}\subset {\mathscr O}(\Omega_0)$ such that 
				${g}^*(\xi,Y)=\sum_{n=1}^{\infty}a_n(Y)\xi^{n-1}$. We have 
			\begin{equation*}
				\begin{aligned}
				\xi^{1-\kappa}{g}^*(\xi,Y)   =\sum_{k\in {\mathbb N}^{n}}
			\widehat{f}_k(\xi)Y^k=\sum_{n=1}^{\infty}a_n(Y)\xi^{n-\kappa}. 
				\end{aligned}
			\end{equation*}
				in $\{0<|\xi|<r\}\times \Omega_0$, which means $\widehat{f}(\xi,Y)=\xi^{1-\kappa}{g}^*(\xi,Y)$
				and $f(x,Y) \in \mathscr O_{\{1/\kappa\}}(S_{0}(I_{\epsilon})\times 
				\Omega_0)$. 
		\end{proof}
			Let ${\phi}_{i}^*(\xi,Y)\in {\mathscr O}(S_0({I}^*)\times U)\; (i=1,2)$.		
			The $\kappa$-convolution is defined by
		\begin{equation}
			 ({\phi}_{1}^*\underset{\kappa}{*}{\phi}_{2}^*)(\xi,Y) 
			 =\int_{0}^{\xi} {\phi}_{1}^*( (\xi^{\kappa}-\eta^{\kappa})^{1/\kappa},Y)
			 {\phi}_{2}^*(\eta,Y)d\eta^{\kappa}.
		\end{equation}
			The following lemma is used for calculations and estimates of convolution equations later. 
		\begin{lem}
			Suppose that ${\phi}_{i}^*(\xi,Y)\in {\mathscr O}(S_0({I}^*)\times U)\; (i=1,2)$ satisfy
			\begin{equation}
				\begin{aligned}
					|{\phi}_{i}^*(\xi,Y)|\leq \frac{C_i |\xi|^{s_i-\kappa} e^{c|\xi|^{\kappa}}}
							{\Gamma(s_i/\kappa)} \quad (s_i>0) \quad for \; (\xi,Y)\in S_0({I}^*)\times U.
				\end{aligned} \label{prod-conv}
			\end{equation}
				Then $({\phi}_{1}^*\underset{\kappa}{*}{\phi}_{2}^*)(\xi,Y) \in 
				{\mathscr O}(S_0({I}^*),U)$ and 
			\begin{equation}
				\begin{aligned}
					|({\phi}_{1}^*\underset{\kappa}{*}{\phi}_{2}^*)(\xi,Y)|
					\leq \frac{C_1C_2 |\xi|^{s_1+s_2-\kappa} e^{c|\xi|^{\kappa}}}
							{\Gamma((s_1+s_2)/\kappa)} .
				\end{aligned} \label{prod}
			\end{equation}
		    \label{lem-prod}
		\end{lem}
		\begin{proof}
				Let $\arg \xi=\theta$. Then it holds that
			\begin{equation*}
				\begin{aligned}				
				 ({\phi}_{1}^*\underset{\kappa}{*}{\phi}_{2}^*)(\xi,Y) 
				 & =\int_{0}^{|\xi|e^{i\theta}} {\phi}_{1}^*( (\xi^{\kappa}-\eta^{\kappa})^{1/\kappa},Y) 
				{\phi}_{2}^*(\eta,Y)d\eta^{\kappa}\\.
				 & =\int_{0}^{|\xi|} {\phi}_{1}^*( (|\xi|^{\kappa}-r^{\kappa})^{1/\kappa}e^{i\theta},Y) 
			 	{\phi}_{2}^*(re^{i\theta},Y)e^{i\kappa\theta}dr^{\kappa}
			 \end{aligned}
		\end{equation*}
			and
		\begin{equation*}
			\begin{aligned}\ &
			 |{\phi}_{1}^*( (|\xi|^{\kappa}-r^{\kappa})^{1/\kappa}e^{i\theta},Y) 
			 	{\phi}_{2}^*(re^{i\theta},Y)e^{i\kappa\theta}|\\  &\leq \frac{C_1C_2}
			 	{\Gamma(s_1/\kappa)\Gamma(s_2/\kappa)}
			 	(|\xi|^{\kappa}-r^{\kappa})^{s_1/\kappa-1}e^{c(|\xi|^{\kappa}-r^{\kappa})} 
			 	|r|^{s_2-\kappa} e^{cr^{\kappa}}. 
			\end{aligned}
		\end{equation*}
		   We have \eqref{prod} from
		\begin{equation*}
		   	\begin{aligned}
		   		\int_{0}^{|\xi|}(|\xi|^{\kappa}-r^{\kappa})^{s_1/\kappa-1}|r|^{s_2-\kappa}dr^{\kappa} =
		   		\frac{\Gamma(s_1/\kappa)\Gamma(s_2/\kappa)}{\Gamma((s_1+s_2)/\kappa)}
		   		|\xi|^{s_1+s_2-\kappa}.
		   	\end{aligned}
		\end{equation*}
		\end{proof}
			We note that Lemma\ref{lem-prod} holds for an infinite $S(I^*)$.		
			Let $\widehat{\phi}_i(\xi,Y)$ be $\kappa$-Borel transform of $\phi_i(x,Y)\; (i=1,2)$. 
			Then    
		\begin{equation}
			\begin{aligned}
			\phi_1(x,Y)\phi_2(x,Y)={\mathcal L}_{\kappa}(\widehat{\phi}_1\underset{\kappa}{*}\widehat{\phi}_2)
			\end{aligned} \label{prod'}
		\end{equation}
			holds. This means $\phi_1(x,Y)\phi_2(x,Y)$ is $\kappa$-Laplace transform of 
			$(\widehat{\phi}_1\underset{\kappa}{*}\widehat{\phi}_2)$.
   \section{\normalsize Nonlinear equation with irregular singularity I } 
	  	First we study the case $F_0(x)\equiv 0$.
	  	The case $F_0(x)\not \equiv 0$ is considered in section 4. Let
		\begin{equation}
			\begin{aligned}\ &
				x^{1+\gamma}\frac{dY}{dx}= A(x)Y+F(x,Y) 
				\quad	Y={}^t(y_1.y_2,\cdots,y_n)  \\
				& A(x)=(a_{i,j}(x))_{1\leq i,j \leq n},\; F(x,y)={}^t(f_1(x,Y),\cdots,f_n(x,Y)),
				\end{aligned} \label{Eq-0}
		  \end{equation}
		  where $\gamma $ is a positive integer. $a_{i,j}(x)$ is holomorphic in $\{|x|<r\}$ and $f_i(x,Y)$ is 
		  holomorphic in $\{(x,Y)\in \mathbb{C}\times \mathbb{C}^n;|x|<r, |Y|<R\}$ with $f_i(x,Y)=O(|Y|^2)$.
		  If $F(0,Y)\not =0$, let $y_i=x z_i$. Then $x^{-1}F(x,xZ)=O(x)$ and 
		 \begin{equation*}
		 	\begin{aligned}
			 	x^{1+\gamma}\frac{dZ}{dx}=& (A(x)-x^{\gamma}I)Z+x^{-1}F(x,x Z). 
		 	\end{aligned} 
		 \end{equation*}
			 Hence we replace $A(x)-x^{\gamma}I$ by $A(x)$
			and $x^{-1}F(x,xZ)$ by $F(x,Z)$. We have 
			 $F(0,Z)=0$. Thus we study, by denoting $Z$ by $Y$ again,      
 		\begin{equation}\left \{
				\begin{aligned} \ &
				x^{1+\gamma}\frac{dy_i}{dx}= \sum_{j=1}^{n}a_{i,j}(x)y_j+f_i(x,Y)
				\quad  i=1,2,\cdots, n, \\
				& A(x)=(a_{i,j}(x)), \; F(x,Y)=(f_{1}(x,Y),\cdots, f_{n}(x.Y))
			\end{aligned}\right. \label{Eq-0''}
		\end{equation}
			 with $f_i(0,Y)=0$. 
		  Let $\{\lambda_{i}\}_{1\leq i \leq n}$ be eigenvalues of $A(0)$. We assume 
			\begin{itembox}[l]{\bf Condition 0. Eigenvalues are distinct} 
				${\mathbf \Lambda}=\{\lambda_{i}; i=1,\cdots,n \}$, \; $\lambda_{i}\not=\lambda_j$
				for $i\not=j$. 
			\end{itembox}
			Set $	\omega_{i,j}=\arg(\lambda_i-\lambda_j)\; (0\leq \omega_{i,j}<2\pi)$ for $i\not=j$ and 	
		\begin{equation}\left \{
			\begin{aligned}\ &
			     {\mathbf  \Lambda}^{\sharp} =\{\lambda_i-\lambda_j; i,j=1,2,\cdots n, i\not=j\}\\
				 &	\theta_{i,j,\ell}=(\omega_{i,j}+2\pi \ell)/\gamma,\; 
		 			\ell \in \mathbb{Z},\\ 
		 			& {\mathbf \Theta_1}=\{\theta_{i,j,\ell};\; i\not =j,\;\ell \in \mathbb{Z}\}
				\end{aligned} \right.\label{not-2}
		\end{equation}
		    Let $\theta_* \not \in {\mathbf \Theta_1}$. Then there exist $\epsilon_*>0$ such that 
		    $(\theta_*-\epsilon_*,\theta_*+ \epsilon_*)\cap  {\mathbf \Theta_1}=\emptyset$. Let 
	     $\delta_*=\pi/2\gamma+ \epsilon_*$ and $I=(\theta_*-\delta_*,\theta_*+\delta_*)$.
       Then the following Proposition holds under Condition 0.
		\begin{prop} Let $\theta_* \not \in {\mathbf \Theta_1}$ and 
		$a_{i,j}(x)\in {\mathscr O}_{\{1/\gamma\}}(S_0(I))$.
		 	  Then there exists a matrix $P(x)$ with elements in ${\mathscr O}_{\{1/\gamma\}}(S_0(I_{\epsilon}))$
		 	  for any small $\epsilon>0$ 
		 	  and $P(0)=Id$ such that $Y=P(x)Z$ transforms 	$x^{1+\gamma}\frac{dY}{dx}= A(x)Y$ to 
		\begin{equation}
			\begin{aligned}
			x^{1+\gamma}\frac{dZ}{dx}= \Lambda(x)Z, 		  
			\end{aligned} \label{eq-L}
		\end{equation}
		where 
			$\Lambda(x)=diag.(\lambda_1(x),\lambda_2(x),\cdots, \lambda_n(x))$ is a diagonal matrix 
			and 
			$\lambda_{i}(x)$ is a polynomial with degree $\leq \gamma$ and $\lambda_{i}(0)=\lambda_i$.
						\label{prop-1}
		\end{prop}
			This proposition follows from multisummablity of the fundamental solution of a system of equations
			(\cite{BBRS}). We give a simpler proof in section 6 under Condition 0. 
			We take $\theta_*\not \in {\mathbf \Theta_1}$ later so that it satisfies other conditions. \par
			Hence we study the following system of nonlinear differential equations:
		 \begin{equation}\tag{Eq-Y} \left \{
		 	\begin{aligned}
		 		Y=& {}^t(y_1.y_2,\cdots,y_n) \\ 	
			 	x^{1+\gamma}\frac{dY}{dx}=& \Lambda(x)Y+F(x,Y) \\
			 	  \Lambda(x)=& diag. (\lambda_{1}(x),\cdots,\lambda_{n}(x)) \\  
			 	 F(x,Y)=& {}^t(f_1(x,Y),\cdots, f_n(x,Y)),
		 	\end{aligned}\right . \label{Eq-Y}
		 \end{equation}
		 where 
		 \begin{equation}
				\begin{aligned}\ &
						\{f_i(x,Y)\}_{1\leq i \leq n}\subset \mathscr O_{\{1/\gamma\}}(S_0(I)\times \Omega),
				\end{aligned}
			\end{equation}
			with $F(x,Y)=O(|Y|^2)$ and $F(0,Y)=0$, and $\{\lambda_{i}(x)\}_{1\leq i \leq n}$ are 
			polynomials with degree $\leq \gamma$ and $\lambda_{i}(0)=\lambda_i$,  \par
       We proceed to give other conditions on the eigenvalues. Let 
			    $\emptyset \not = {\mathbf \Lambda'}\subset {\mathbf \Lambda}$ and  
		  	  ${\mathbf \Lambda'}=\{\lambda_{i}; i=1,\cdots,n' \}$.
		  	We assume the following two conditions on ${\mathbf \Lambda'}$.
  	\begin{itembox}[l]{\bf Condition 1. Partial Poincar{\`e} cndition} 
  			There exist $0 \leq { \theta_{\mathbf \Lambda'}}<2\pi$ and 
				$0<\delta_{\mathbf \Lambda'}<\pi/2$ such that ${\mathbf \Lambda'}
				\subset \Sigma= \{\eta \not=0; |\arg \eta - \theta_{\mathbf \Lambda'}|
				<\delta_{\mathbf \Lambda'} \}$.
		\end{itembox}
		\begin{itembox}[l]{	\bf Condition 2. Partial non resonance}
		  \begin{equation}
		  	\sum_{j=1}^{n'} \lambda_{j}m_{j}-\lambda_i\not =0 \;\; \mbox{for $m \in {\mathbb N}^{n'}$ 
		  	with $ |m|\geq 2$},\; 1\leq i \leq n.
		 	\end{equation}
	  \end{itembox}
			 	Let  
		\begin{equation}
		  	\begin{aligned}
		  	  \frak{L} =& \bigcup_{i=1}^n \big\{\sum_{j=1}^{n'} \lambda_{j}m_{j}-\lambda_i;\; |m|\geq 2\big \}
		  	  \not \ni 0.
				\end{aligned}\label{fL}
		\end{equation}
		    Function $\gamma\eta+\sum_{j=1}^{n'} \lambda_{j}m_{j}-\lambda_i\; (|m|\geq 2)$ 
			  does not vanish for $\gamma\eta \not \in -{\frak{L}}$.
			  Let $L(\theta)=
			  \{r \geq  0; re^{i\theta}\}$ be a half line in a direction $\theta$.
		\begin{lem}
				There exist an interval $\widehat{J}=(\widehat{\theta}_0-\widehat{\epsilon}_0,
				\widehat{\theta}_0+\widehat{\epsilon}_0)$ $(\widehat{\epsilon}_0>0)$ and 
				constants $r_0,C_{\widehat{J}}>0$ 
			  such that $ S(\widehat{J})\cap (-\overline{\frak L})=\emptyset$ and
	 		\begin{equation}
	 			\begin{aligned}
			 		|\gamma \eta+	\sum_{j=1}^{n'} \lambda_{j}m_{j}-\lambda_i|
			 		\geq C_{\widehat{J}}(|\eta|+|m|), \;\; |m| \geq 2, \;\; 1\leq i \leq n
	 			\end{aligned} 	 		\end{equation}\label{lem-0}	
	 			for $\eta \in S(\widehat{J})\cup\{|\eta|<r_0 \}$.  
	 	\end{lem}
	 	\begin{proof}
	     	Let $\varepsilon>0$ be a small constant with 
	     	$\delta_{\mathbf \Lambda'}+\varepsilon<\pi/2$ and
	 			$\Sigma_{\varepsilon}=\{\eta \not=0; |\arg \eta - \theta_{\mathbf \Lambda'}|
	 			<\delta_{\mathbf \Lambda'}+\varepsilon \}$. Then $ \frak{L}\cap \Sigma_{\varepsilon}^{c}$ and 
	 			$(-\frak{L})\cap (-\Sigma_{\varepsilon}^{c})$ are finite. Hence
	 			there exists $\widehat{\theta}_0$ with $L(\widehat{\theta}_0) \subset (-\Sigma_{\varepsilon}^{c})$ 
	 			and $\widehat{\epsilon}_0>0$ 
	 			such that 
	 			$L(\theta)\cap (-\overline{\frak L})=\emptyset$ for
	 			$\theta\in \widehat{J}=(\widehat{\theta}_0-\widehat{\epsilon}_0,\widehat{\theta}_0+
	 			\widehat{\epsilon_0})$ and 			
	 		\begin{equation*}
	 				|\gamma \eta
			 		+	\sum_{j=1}^{n'} \lambda_{j}m_{j}-\lambda_i| \geq C_{\widehat{J}}'
			 		(|\eta|+|m|)  \quad \eta \in S(\widehat{J}).
	 		\end{equation*}
	 		holds for some constant  $C_{\widehat{J}}'>0$.  Since there exist $r_0,c_0>0$ such that
	 		$|\gamma \eta+	\sum_{j=1}^{n'} \lambda_{j}m_{j}-\lambda_i| >c_0$	in $\{|\eta|<r_0\}$, 
	 		the assertion holds.	 			 
	 	\end{proof}
	  	In addition we assume a condition of an interval $\widehat{I}$ to define a sector $S_0(I)$ 
	  	in $x$-space.  
	  \begin{itembox}[l]{	\bf Condition 3.}
	 		{\it Let  $\widehat{I}=(\theta_*-{\epsilon}_*,\theta_*+ {\epsilon}_*)\;\; ({\epsilon}_*>0)$ with
				$L(\gamma\theta)\cap \mathbf{\Lambda}^{\sharp}=\emptyset$ for $\theta\in \widehat{I}$ and  
 					there exist $ C_{\widehat{I}},r>0$ such that
			\begin{equation}
					\begin{aligned}\ &
	 				|\gamma \xi^{\gamma}
			 		+	\sum_{j=1}^{n'} \lambda_{j}m_{j}-\lambda_i| \geq C_{\widehat{I}}
			 		(|\xi|^{\gamma}+|m|) \\ & \;|m|\geq 2 \quad    
			 		\xi \in S(\widehat{I})\cup\{|\xi|<r\}, \quad 1\leq i \leq n,
			 		\end{aligned}
			 		\label{Ine}
	 		\end{equation}
	 		holds.}
	 	\end{itembox}	
	 	    The interval $\widehat{I}$ appears in \eqref{L}. 
		    We show the existence of $\widehat{I}$ satisfying Condition 3. 
		 		Let $\widehat{J}$ be that in Lemma \ref{lem-0}. 
		 	 		Since $\mathbf{\Lambda}^{\sharp}$ is finite, we can take $\widehat{J}$ such that
	 		$S(\widehat{J})\cap \mathbf{\Lambda}^{\sharp}=\emptyset$. Hence 
	 		$S(\widehat{J})\cap \big((-\overline{\frak L})\cup \mathbf{\Lambda}^{\sharp} \big)=\emptyset$, 
		 	Let $\theta_*=\widehat{\theta}_0/\gamma$, ${\epsilon}_*=\widehat{\epsilon}_0/\gamma$ and 
		 	$\widehat{I}=(\theta_*-{\epsilon}_*,\theta_*+ {\epsilon}_*)$.    
      Then $\gamma \widehat{I}=\widehat{J}$ and if $\xi \in S(\widehat{I})$, 
       $\xi^{\gamma} \in S(\widehat{J})$. This $\widehat{I}$ satisfies the above condition. 
		    We note that
		 		$L(\gamma\theta)\cap \mathbf{\Lambda}^{\sharp}=\emptyset$ is equivalent to
		 		$\theta \not \in \mathbf{\Theta_1}$, hence 	$\widehat{I}\cap\mathbf{\Theta_1}=\emptyset$.
		\begin{rem}	
			We can choose $\theta_{\mathbf \Lambda'}=\widehat{\theta}_0$, 
			by changing ${\delta}_{\mathbf \Lambda'}$. if necessary. 
		\label{rem2.3}
	\end{rem}	  
			 Let us define an interval $I$ for $\widehat{I}$ 
			under Condition 3. Let $\delta_*=\pi/2\gamma+ {\epsilon}_*$ and 
			$I=(\theta_*-\delta_*,\theta_*+\delta_*)$. Let us consider a sector $S_0(I)$ with angle $>\pi/\gamma$
			in $x$-space. 
			Then we get one of the main theorems.
	 \begin{thm} 
	 		   There exists $\Phi(x,Z)=(\phi_1(x,Z),\cdots,\phi_{n}(x,Z))$\;
	   		 such that for any small $\epsilon>0$ there exists $r_{\epsilon}>0$,
	   		$\phi_i(x,Z)\in {\mathscr O}_{\{1/\gamma\}}
	   		(S_{\{0\}}(I_{\epsilon})\times \{Z=(z_1, \cdots,z_{n'})\in {\mathbb C}^{n'}; |Z|<r_{\epsilon} \})$
	   		and the followings hold.
	   		\begin{enumerate}
	   		\item[{\rm(1)}] $\phi_{i}(x,Z)=z_i+O(|Z|^2)$ for $1\leq i\leq n'$ and 
	   				$\phi_{i}(x,Z)=O(|Z|^2)$ for $i > n'$. 
   		 \item[{\rm(2)}] Let $\mathcal S$
   		  be an open set in $S_{\{0\}}(I_{\epsilon})$ and 
   			$Z(x)=(z_1(x),\cdots,z_{n'}(x))\; (x\in \mathcal S,|Z(x)|<r_{\epsilon})$ be a solution of
	 				\begin{equation}
	 					\begin{aligned}
	 					x^{1+\gamma}\frac{dz_{i}}{dx}=\lambda_{i}(x)z_i,\; i=1,2,\cdots,n'.
	 					\end{aligned}
	 				\end{equation}	
	 			Then $Y(x)=\Phi(x,Z(x))$, $y_i(x)=\phi_{i}(x,z_1(x),\cdots,z_{n'}(x))\; (1\leq i \leq n)$, 
	 			satisfies \eqref{Eq-Y} in  $\mathcal S$.
	 		\end{enumerate}
   		   \label{main-th}
 	 		\end{thm}
	 		\begin{rem}	{\rm (1)}\; $\Phi(x,Z)$ is determined by solving a system of partial differential 
	 		equations \eqref{Eq-Phi} in section 3 and depends on choice of ${\theta}_*$. \label{rem-2}  \\
	 				{\rm (2)} Theorem \ref{main-th} means that there exist solutions of \eqref{Eq-Y} 
	 				with exponential series, called often transseries {\rm (see also Remarks 3.4)},  
				\begin{equation}	\left \{
					\begin{aligned}\ &
				 z_i(x)=A_i\exp(\int^x\frac{\Lambda_i(\tau)}{\tau^{\gamma+1}}d\tau )\;\; (1\leq i\leq n'),\\
				 & y_i= z_i(x)+ \sum_{|p| \geq 2}C_{i,p}(x)Z(x)^p   \;\; (1\leq i\leq n'), \\
				 & y_i= \sum_{|p| \geq 2}C_{i,p}(x)Z(x)^p   \;\; (i> n'),\quad 
				 C_{i,p}(x)\in {\mathscr O}_{\{1/\gamma\}}(S_{\{0\}}(I_{\epsilon})).
				\end{aligned}\right .
			\end{equation}
			{\rm (3)} If $\Lambda'=\{\lambda_1\}, \lambda_1\not=0$, then Condition 1 is obviously holds and
			Condition 2 is $m_1\lambda_1-\lambda_i\not=0$ for $m_1\geq 2$ and $i=2,\cdots,n$. 
				\end{rem}
\section{\normalsize Proof of Theorem \ref{main-th} }	
	 		Our assumptions are  
	 	\begin{equation}\left \{
	 		\begin{aligned}
	 			 		 \ &
		 		\{f_{i}(x,Y)\}_{1\leq i \leq n}\subset \mathscr O_{\{1/\gamma\}}
		 		(S_0(I)\times \{|Y|<R\}), \\
		 		 &I=({\theta}_*-\delta_*,{\theta}_*+\delta_*), 
		 		 \widehat{I}=({\theta}_*-{\epsilon}_*,{\theta}_*+{\epsilon}_*),\;
		  		 \delta_*=\pi/2\gamma+{\epsilon}_*, \\ &
		 		  L(\gamma \theta)\cap \mathbf \Lambda^{\sharp}=\emptyset, \; for \;\theta\in \widehat{I}
	 		\end{aligned} \right . \label{f_i}
	 	\end{equation}
		 	with	$f_{i}(0,Y)=0$, $f_{i}(x,Y)=O(|Y|^2)$ and \eqref{Ine} holds.
	\subsection{\normalsize Construction of $\Phi(x,Z)$-I }	 	
	  	In order to construct $\Phi(x,Z)$ in Theorem \ref{main-th} we introduce an auxiliary
  		system of nonlinear partial differential equations
		\begin{equation}\tag{Eq-$\Phi$} 
				\left \{
			\begin{aligned}\ &
				\Phi(x,Z)=(\phi_1(x,Z),\phi_2(x,Z),\cdots, \phi_n(x,Z))\\ &		
	  		x^{1+\gamma}\frac{\partial \phi_i}{\partial x}+
	  		\sum_{j=1}^{n'} \lambda_{j}(x)z_j	\frac{\partial \phi_i}{\partial z_j}	  	
	    	-\lambda_{i}(x){\phi}_i=f_i(x,\Phi) \;\; 1\leq i \leq n \\ &
	    	(x,Z)=(x,z_1,\cdots,z_{n'})\in {\mathbb C\times \mathbb C^{n'}},
	 		\end{aligned}\right . \label{Eq-Phi}
	 	\end{equation} 
				 A similar type equation appeared in \cite{Ouchi-0}.    
	    	Assume we find a nice solution $\Phi(x,Z)$ of \eqref{Eq-Phi}. Let  
	    	$Z(x)=(z_1(x), \cdots,z_{n'}(x))$ be a solution of
	  \begin{equation}\tag{Eq-Z}
	    		 	\begin{aligned}
		 		 	x^{1+\gamma}\frac{dz_i}{dx}= \lambda_i(x)z_i \quad 1\leq i \leq n'   \label{Eq-Z}
		 	\end{aligned}
  	\end{equation} 
	    	and $Y(x)=\Phi(x,Z(x))$ ($y_{i}(z)=\phi_i(x,z_1(x),\cdots,z_{n'}(x))$). Then we have 
	     \begin{equation*}
	    	\begin{aligned}
	    		x^{1+\gamma}\frac{dy_i}{dx}& =x^{1+\gamma}\frac{\partial\phi_i}{\partial x}+
	    		x^{1+\gamma}\sum_{j=1}^{n'}\frac{\partial\phi_i}{\partial z_j}\frac{dz_j}{dx}\\
	    		&=x^{1+\gamma}\frac{\partial\phi_i}{\partial x}+
	    			\sum_{j=1}^{n'}\lambda_{j}(x)z_{j}\frac{\partial\phi_i}{\partial z_j}
	    		=\lambda_{i}(x){\phi}_i+f_i(x,\Phi)
	    		\\&= \lambda_{i}(x){y}_i+f_i(x,Y(x))	
	      \end{aligned}
	  \end{equation*}
	  		and $Y(x)$ will be a solution of \eqref{Eq-Y}. \par
		    	We construct $\Phi(x,Z)$ as follows. Let $\Psi(Z)=( \psi_1(Z),\psi_2(Z),\cdots,\psi_n(Z))
	    	=(z_1,z_2,\cdots,z_{n'},0,\cdots,0)$ and $\Phi(x,Z)=U(x,Z)+\Psi(Z)$. 
				We change \eqref{Eq-Phi} to a system of equations of $U(x,Z)=(	u_1(x,Z),\cdots,	u_n(x,Z))$.
				By $\big (\lambda_{i}(x)-\sum_{j=1}^{n'} \lambda_{j}(x)z_j
	  			\frac{\partial }{\partial z_j}\big)\psi_i(Z)=0$ we have
		\begin{equation}\tag{Eq-U}
			\begin{aligned}\ & \big(
	  			x^{1+\gamma}\frac{\partial }{\partial x}+
	  			\sum_{j=1}^{n'}\lambda_{j}(x)z_j\frac{\partial}{\partial z_j}-\lambda_{i}(x)\big)u_i
	  			=f_i(x,U+\Psi(Z)).
	    \end{aligned} \label{Eq-U0}
	  \end{equation}
	  	Let $f_i(x,\Phi)=\sum_{\begin{subarray}\ m\in \mathbb N^{n}, |m|\geq 2 \end{subarray}}	
	  	f_{i,m}(x)\Phi^m$. Then there exist 
	  	$g_{i,k,\ell}(x)\in \mathscr{O}_{\{1/\gamma\}}(S_{0}(I))$, \; 
	  	$(k,\ell)\in \mathbb N^{n'}\times \mathbb N^{n}$,  such that 
	  \begin{equation*}
	  	\begin{aligned}\ f_{i}(x,U+\Psi(Z))= &
	  		\sum_{\begin{subarray}\ m\in \mathbb N^{n}, |m|\geq 2 \end{subarray}}	
	  		f_{i,m}(x)(U+\Psi(Z))^m \\
			= &  \sum_{\begin{subarray}\ (k,\ell)\in \mathbb N^{n'}\times \mathbb N^{n},		 
							\\ |k|+|\ell|\geq 2, \ell\not=0 
								\end{subarray}}  g_{i,k,\ell}(x)Z^k U^{\ell}+ f_{i}(x,\Psi(Z))
	  	\end{aligned}
	  \end{equation*}
	    with $g_{i,k,\ell}(0)=0$ ($\because f_i(0,Y)=0$) and 
	    $|f_{i}(x,\Psi(Z))|\leq M|x||Z|^2$. 
		Let 	 
	 	\begin{equation}
	 		\begin{aligned} \ &
	 		L=	x^{1+\gamma}\frac{\partial} {\partial x}+
	  	\sum_{j=1}^{n'} \lambda_{j}z_j	\frac{\partial}{\partial z_j}-\lambda_{i}
	  	\quad \lambda_i=\lambda_{i}(0), \\
	  	& 	\lambda_i^{*}(x)=	\lambda_i(x)-\lambda_i,\;\;
	  	h_i(x,Z)= f_{i}(x,\Psi(Z))=\sum_{|p|\geq 2}h_{i,p}(x)Z^p.  	
	  	\end{aligned}
		\end{equation}
			Then $	\lambda_i^{*}(0)=0$ and \eqref{Eq-U0} is 
	 \begin{equation}
	 	\begin{aligned}
	 	Lu_i= &-(\sum_{j=1}^{n'} \lambda_{j}^*(x)z_j	\frac{\partial }{\partial z_j}	  
	 		    -\lambda_{i}^*(x))u_i \\ +&
	 	   \sum_{\begin{subarray}\ (k,\ell) \in \mathbb N^{n'}\times \mathbb N^{n},
							\\ |k|+|\ell|\geq 2,\ell\not=0
								\end{subarray}} 
						g_{i,k,\ell}(x)Z^k U^{\ell}+ h_i(x,Z).
	 	\end{aligned} \label{Eq-U''}
	 \end{equation}
		We introduce an auxiliary  parameter $\varepsilon$ in order to show successive process 
		of construction of a solution clearly,
		 \begin{equation}\tag{Eq-U$_{\varepsilon}$}
	 	\begin{aligned}
	 	Lu_i
	 		= &-\varepsilon(\sum_{j=1}^{n'} \lambda_{j}^*(x)z_j	\frac{\partial}{\partial z_j}
	 	 		    -\lambda_{i}^*(x))u_i \\ 
	 	  &	+ \varepsilon\sum_{\begin{subarray}\ (k,\ell)\in \mathbb N^{n'}\times \mathbb N^{n},
							\\ |k|+|\ell|\geq 2, \ell\not=0
								\end{subarray}} 
						g_{i,k,\ell}(x)Z^k U^{\ell}+ \varepsilon h_i(x,Z). 
	 	\end{aligned} \label{Eq-e}
	 \end{equation}
	 	 If $\varepsilon=1$, \eqref{Eq-e} coincides with \eqref{Eq-U0} and \eqref{Eq-U''}.
	 	  Our purpose is to find 
		 a solution $U(x,Z,\varepsilon)=(	u_1(x,Z,\varepsilon),\cdots, 	u_n(x,Z,\varepsilon))$
		 of \eqref{Eq-e}. It is constructed as follows. 
	 		Let 
	 \begin{equation}
	 	\begin{aligned}
		 	u_i(x,Z,\varepsilon)=\sum_{\begin{subarray}\ (p,q) \in \mathbb N^{n'}\times \mathbb N
		 				\\ \;\; q \geq 1 \end{subarray}} 
							C_{i,p,q}(x)Z^p\varepsilon^{q}, \quad 1\leq i \leq n,
		 	\end{aligned}
	 \end{equation}
		 and note $U^{\ell}=\prod_{s=1}^{n}u_s^{\ell_s}\; \ell=(\ell_1,\cdots,\ell_s, \cdots,\ell_n)$ and
	 \begin{equation*}
	 	\begin{aligned}
		 	u_s^{\ell_s}=\prod_{j=1}^{\ell_s} \big(\sum_{\begin{subarray}\ (p^{s,j},q^{s,j}) 
		 	\in \mathbb N^{n'}\times \mathbb N
		 				\\ \;\; q^{s,j} \geq 1 \end{subarray}} 
							C_{s,p^{s,j},q^{s,j}}(x)Z^{p^{s,j}}\varepsilon^{q^{s,j}}\big).
	 	\end{aligned}
	 \end{equation*}
	   By substituting $u_{i}(x,Z,\varepsilon)$ into \eqref{Eq-e} and comparing the coefficient of
	   $Z^p{\varepsilon}^q$, we get 
	 \begin{equation}\tag{Eq-C}
	 		\begin{aligned}\	& 
		 	(x^{1+\gamma}\frac{d}{dx}+\sum_{j=1}^{n'} p_j\lambda_j-\lambda_i)C_{i,p,q}(x)=
		 	-\big(\sum_{j=1}^{n'} p_j\lambda_{j}^*(x)-\lambda_{i}^*(x)\big)C_{i,p,q-1}(x)
				\\ & +
			\sum_{\begin{subarray}\ (k,\ell)\in \mathbb N^{n'}\times \mathbb N^{n} 
							\\ |k|+|\ell|\geq 2, \ell \not=0 \end{subarray}}  
							g_{i,k,\ell}(x) \Big(\sum_{\begin{subarray}\
							\sum_{s=1}^n\big( \sum_{j=1}^{{\ell}_s} p^{s,j}\big)+k=p \\
							\sum_{s=1}^n\big( \sum_{j=1}^{{\ell}_s} q^{s,j}\big)+1=q	
								\end{subarray}}
		  		\prod_{j=1}^{\ell_1} C_{1,p^{1,j},q^{1,j}}(x)
			\\ &	\times		\prod_{j=1}^{\ell_2} C_{2,p^{2,j},q^{2,j}}(x)
								\cdots \cdots
						\prod_{j=1}^{\ell_n} C_{n,p^{n,j},q^{n,j}}(x) \Big)+
			 		\delta_{q,1}h_{i,p}(x),
	 	\end{aligned} \label{Eq-Cv}
	 \end{equation}
	 	 		where  $k, p^{s,j},p \in \mathbb N^{n'}$, $q,q^{s,j}\in \mathbb N$ and $\delta_{i,j}$ is Kronecker's 
	 	 		delta.
\subsection{\normalsize Construction of $\Phi(x,Z)$-II}		  			
 		We try to construct $C_{i,p,q}(x)$ by Laplace integral
	 \begin{equation}
	 		C_{i,p,q}(x)=
	 			\int_{0}^{\infty e^{\sqrt{-1}\theta} }e^{-(\frac{\xi}{x})^{\gamma}}
	 		\widehat{C}_{i,p,q}(\xi)d\xi^{\gamma}\quad \theta \in \widehat{I}	 	\label{L}	
	 \end{equation}
		 We use the following notation 
		 \begin{equation*}
		 	 	 W_{1}(\xi)\underset{\gamma}{*}W_2(\xi)\underset{\gamma}{*}\cdots
	  	\underset{\gamma}{*}W_{N}(\xi)
	  	=\underset{*\gamma}{\underbrace{\prod_{i=1}^{N}W_{i}(\xi)}}.
		 \end{equation*}
		By using \eqref{prod'},	we get the following convolution equations 
		from \eqref{Eq-Cv},  
	 \begin{equation}\tag{Eq-$\widehat{C}$}
	 \begin{aligned}\ &
	 		(\gamma \xi^{\gamma}+\sum_{j=1}^{n'} p_j\lambda_j-\lambda_i)	\widehat{C}_{i,p,q}(\xi)
	 		= -\big(\sum_{j=1}^{n'} p_j\widehat{\Lambda_{j}^*}(\xi)-
			\widehat{\Lambda_{i}^*}(\xi)\big)\underset{\gamma}{*}\widehat{C}_{i,p,q-1}(\xi)
			\\ &  
		  +\sum_{\begin{subarray}\ (k,\ell)\in \mathbb N^{n'}\times \mathbb N^{n} 
							\\ |k|+|\ell|\geq 2, \ell\not=0\end{subarray}}  
							\widehat{g}_{i,k,\ell}(\xi)\underset{\gamma}{*}\Big( \sum_{\begin{subarray}\
							\sum_{s=1}^n( \sum_{j=1}^{\ell_{s}} p^{s,j})+k=p \\
							\sum_{s=1}^n( \sum_{j=1}^{\ell_{s}} q^{s,j})+1=q	
								\end{subarray}}
			 \\ & \quad	
			 	\underset{*\gamma}{\underbrace{\prod_{j=1}^{\ell_1}\widehat{C}_{1,p^{1,j},q^{1,j}}(\xi)
						\prod_{j=1}^{\ell_2}\widehat{C}_{2,p^{2,j},q^{2,j}}(\xi)
								\cdots \cdot
						\prod_{j=1}^{\ell_n}\widehat{C}_{n,p^{n,j},q^{n,j}}(\xi) 
						}}\Big) +\delta_{q,1}\widehat{h}_{i,p}(\xi).
	 	\end{aligned} \label{Eq-c1}
	 \end{equation}
		 The main result of this subsection is Proposition \ref{prop-M} concerning existence and estimate 
		 of $\widehat{C}_{i,p,q}(\xi)$.
		 We proceed to solve \eqref{Eq-c1}. There are 2 steps, 
		 to determine $\widehat{C}_{i,p,q}(\xi)$ and to estimate them. \\
 (I)\;{\it Determination of $\widehat{C}_{i,p,q}(\xi)$}. 
 		 		We notice that there exists a constant $C>0$ such 
			  that for $ \xi \in S(\widehat{I})\cup\{|\xi|<r\}$ and $|p|\geq 2$	
	  	\begin{equation}
	  		\begin{aligned}
	  		|\gamma \xi^{\gamma}+\sum_{j=1}^{n'} p_j\lambda_j-\lambda_i|\geq
	  			C (|\xi|^{\gamma}+|p|). 
	   		\end{aligned}\label{est-inv}
	  	\end{equation}
			   Let $q=1$. Then $\widehat{C}_{i,p,1}=0$ for $|p|\leq 1$ and 
		  \begin{equation}
		  	\begin{aligned}\ &
				  \widehat{C}_{i,p,1}(\xi)=\frac{	\widehat{h}_{i,p}(\xi)}
				  		{(\gamma \xi^{\gamma}+\sum_{j=1}^{n'} p_j\lambda_j-\lambda_i)}
		  	\end{aligned}
		  \end{equation}
		    for $|p|\geq 2$. $ \xi^{\gamma-1}\widehat{C}_{i,p,1}(\xi)$ is holomorphic at $\xi=0$. 
		    Assume $ \{ \widehat{C}_{j,r,s}(\xi) \}_{j=1}^n$  $(s<q) $ are determined 
		    such that $ \widehat{C}_{j,r,s}(\xi)=0$ for $|r|\leq 1$. 
		    We denote the right hand side of \eqref{Eq-c1} by
		 \begin{equation}
		 		\begin{aligned}\ &
				 	{\mathcal F}_{i,p,q}\big(\xi, \{\widehat{C}_{j,r,s}\}_{j=1}^n, (r,s)\in \Sigma_{p,q}\big), \\
				 	 & \Sigma_{p,q}=\{(r,s);2 \leq |r|\leq |p|, 1\leq s \leq q-1 \}.
		 	 \end{aligned} 
		 \end{equation}	
	 	 	 Then  $\{\widehat{C}_{i,p,q}(\xi)\}_{i=1}^n$  are  determined by
		 \begin{equation}
		  	\widehat{C}_{i,p,q}(\xi)=\frac{	{\mathcal F}_{i,p,q}(\xi,
		  	 \{\widehat{C}_{j,r,s}\}_{j=1}^n, (r,s)\in \Sigma_{p,q})}
		  	{(\gamma \xi^{\gamma}+\sum_{j=1}^{n'} p_j\lambda_j-\lambda_i)}
		  \end{equation}
			  and $\widehat{C}_{i,p,q}(\xi)=0$ for $|p|\leq 1$.
			  Thus $\widehat{C}_{i,p,q}(\xi)\; (|p|\geq 2, q\geq 1)$ are successively determined and they are 
			  holomorphic in  $(\{0 <|\xi|<r\}\cup S(\widehat{I})$. Moreover 
			  {$\xi^{\gamma-1}\widehat{C}_{i,p,q}(\xi)$ is holomrphic at $\xi=0$.  \\
	 (II)\; {\it Estimate of $\widehat{C}_{i,p,q}(\xi)$.}\;
				Let ${\epsilon}>0$ be a small constant. We obtain estimates of $\widehat{C}_{i,p,q}(\xi)$ in 
				a  subsector $S(\widehat{I}_{\epsilon}) \subset S(\widehat{I})$. 
				We often apply Lemma \ref{lem-prod} to
				estimate. We have 
	  \begin{prop}Let $\epsilon>0$ be an arbitrary small constant. Then 
		  	there exist positive constants $r, M_{i,p,q}$ and $c$ depending on ${\epsilon}$ such that
	  	\begin{equation}
	  		\begin{aligned}
		  		|\widehat{C}_{i,p,q}(\xi)|\leq \frac{M_{i,p,q}|\xi|^{q-\gamma}e^{c|\xi|^{\gamma}}}
		  		{\Gamma(q/\gamma)} \quad  \xi \in \{0<|\xi|<r \}\cup S(\widehat{I}_{\epsilon})
	  		\end{aligned}
	  	\end{equation}
		  	and the series $\sum_{\begin{subarray}\ (p,q)\in {\mathbb N^{n'}}\times  {\mathbb N}\\  |p|\geq 2, 
	  							  q\geq 1	\end{subarray}}
	  						 M_{i,p,q}Z^p s^q $ converges in a neighborhood of $(Z,s)=(0,0)
	  						 \in {\mathbb C^{n'}}\times{\mathbb C} $. \label{prop-M}
	  \end{prop}
	  	  Before the proof we note inequality \eqref{est-inv} and 
	  	  that there exist constants $G_{i,k,\ell}, H_{i,p}$ and $c$ such that
	 \begin{equation}\left \{
	 		\begin{aligned}
	 		\ &	 |\widehat{g}_{i,k,\ell}(\xi)|\leq  
	 			 	\frac{G_{i,k,\ell}|\xi|^{1-\gamma}e^{c|\xi|^{\gamma}}}{\Gamma(1/\gamma)} \\ &
	 	    |\widehat{h}_{i,p}(\xi)|\leq 
	 	    	\frac{H_{i,p} |\xi|^{1-\gamma}e^{c|\xi|^{\gamma}}}{\Gamma(1/\gamma)} \qquad 
	 	    	\xi \in  \{0<|\xi|<R \}\cup S(\widehat{I}_{\epsilon}). 
	 	  \end{aligned} \right. 	
	 \end{equation}
			Here $\sum_{\begin{subarray}\ (k,\ell)\in\mathbb{N}^{n'}\times \mathbb{N}^n 
		 \\ |k|+|\ell| \geq 2, \ell \not =0 
			 \end{subarray}} G_{i,k,\ell}Z^k U^{\ell} $	
		 converges in a neighborhood of $(Z,U)=(0,0)\in \mathbb{C}^{n'}\times\mathbb{C}^n$ and 	
		 $ \sum_{|p|\geq 2}H_{i,p} Z^p$ converges in a neighborhood of $Z=0\in \mathbb{C}^{n'}$ 
		 (see Lemma \ref{lem-Cau}). \par  
			{\it Proof of Proposition \ref{prop-M}}.
			 The proof consists of 2 parts, (1) determination of $M_{i,p,q}$ and (2) convergence of
			the series. \\ \; 
		(1) {\it Determination of $M_{i,p,q}$.} \;  Let $0<r<R$. First we show how to determine  
	   		$M_{i,p,q}\; (p\in {\mathbb N}^{n'} \;|p|\geq 2,q\geq 1)$ and study their relations. 	  	 
	  		For $ q=1$ and $|p|\geq 2$ there exist a constant $C>0$ such that
	    \begin{equation*}
	  			\begin{aligned}
	  				|\widehat{C}_{i,p,1}(\xi)|=\frac{	|\widehat{h}_{i,p}(\xi)|}
	  					{|\gamma \xi^{\gamma}+\sum_{j=1}^{n'} p_j\lambda_j-\lambda_i|}
	  				\leq C	H_{i,p}\frac{|\xi|^{1-\gamma}e^{c|\xi|^{\gamma}}}{\Gamma(1/\gamma)}
	  			\end{aligned}
	  	\end{equation*}
    			and take $M_{i,p,1}=CH_{i,p}$. Assume $ \{ M_{i,p,q'} \}_{i=1}^n$  $(q'<q )$ are
	     		determined such that 
	  	\begin{equation}
	  		\begin{aligned}
	  		|\widehat{C}_{i,p,q'}(\xi)|\leq \frac{M_{i,p,q'}|\xi|^{q'-\gamma}e^{c|\xi|^{\gamma}}}
	  		{\Gamma(q'/\gamma)} \quad  \xi \in \{0<|\xi|<r \}\cup S(\widehat{I}_{\epsilon}).
	  		\end{aligned}
	  	\end{equation}
	     Let us notice relation \eqref{Eq-c1}. It follows from Lemma \ref{lem-prod} and \eqref{est-inv} that 
	      there exists a constant $A>0$ such that
	    \begin{equation}
	    	\begin{aligned}
	   	D_1= \frac{\big|\big(\sum_{j=1}^{n'} p_j|\widehat{\Lambda_{j}^*}(\xi)-
			\widehat{\Lambda_{i}^*}(\xi)\big)\underset{\gamma}{*}\widehat{C}_{i,p,q-1}(\xi)\big|}
			{|\gamma \xi^{\gamma}+\sum_{j=1}^{n'} p_j\lambda_j-\lambda_i|}
			 \leq  \frac{AM_{i,p,q-1}|\xi|^{q-\gamma}e^{c|\xi|^{\gamma}}}
	  		{\Gamma(q/\gamma)}.
	 	    	\end{aligned}
	    \end{equation}
	  Let 
	  \begin{equation*}
	  	\begin{aligned}
	  	D_2 & ={|\gamma \xi^{\gamma}+\sum_{j=1}^{n'} p_j\lambda_j-\lambda_i|^{-1}}
	  				\Big |\sum_{\begin{subarray}\ (k,\ell)\in \mathbb N^{n'}\times \mathbb N^{n} 
							\\ |k|+|\ell|\geq 2 ,\ell\not=0\end{subarray}}  
							\widehat{g}_{i,k,\ell}^{*}(\xi)\underset{\gamma}
							{*}\Big( \sum_{\begin{subarray}\
							\sum_{s=1}^n\big( \sum_{j=1}^{\ell_{s}} p^{s,j}\big)+k=p\\
							\sum_{s=1}^n\big( \sum_{j=1}^{\ell_{s}} q^{s,j}\big)+1=q	
								\end{subarray}}
			 \\ & \quad \underset{*\gamma}{\underbrace{
			  		\prod_{j=1}^{\ell_1}\widehat{C}_{1,p^{1,j},q^{1,j}}(\xi)
						\prod_{j=1}^{\ell_2}\widehat{C}_{2,p^{2,j},q^{2,j}}(\xi)
						\cdots \cdots
						\prod_{j=1}^{\ell_n}\widehat{C}_{n,p^{n,j},q^{n,j}}(\xi)}} 
						\Big) \Big|.
	  	\end{aligned}
	  \end{equation*}
	  Since
	  	  \begin{equation*}
	  	\begin{aligned} \ & \Big|
	  	\underset{*\gamma}{
			  	\underbrace{	{\prod_{j=1}^{\ell_1}}\widehat{C}_{1,p^{1,j},q^{1,j}}(\xi)
						\prod_{j=1}^{\ell_2}\widehat{C}_{2,p^{2,j},q^{2,j}}(\xi)
						\cdots \cdots
						\prod_{j=1}^{\ell_n}\widehat{C}_{n,p^{n,j},q^{n,j}}(\xi) 
						\Big |}} \\
					 \leq	& {\prod_{j=1}^{\ell_1}}M_{1,p^{1,j},q^{1,j}}
						\prod_{j=1}^{\ell_2}M_{2,p^{2,j},q^{2,j}}
						\cdots \cdots
						\prod_{j=1}^{\ell_n}M_{n,p^{n,j},q^{n,j}} \\
						& \times \big(\underset{*\gamma}
					{\underbrace{{\prod_{j=1}^{\ell_1}}\frac{|\xi|^{q^{1,j}-\gamma}e^{c|\xi|^{\gamma}}}
						{\Gamma(q_{1,j}/\gamma)}
						\prod_{j=1}^{\ell_2}\frac{|\xi|^{q^{2,j}-\gamma}e^{c|\xi|^{\gamma}}}
						{\Gamma(q_{2,j}/\gamma)}
						\cdots \cdots
						\prod_{j=1}^{\ell_n}\frac{|\xi|^{q^{n,j}-\gamma}e^{c|\xi|^{\gamma}}}
						{\Gamma(q^{n,j}/\gamma)}\big)}},
	  	\end{aligned}
	  \end{equation*}	
	  we have
	    \begin{equation}
	  	\begin{aligned}
	  	D_2 \leq &C \sum_{\begin{subarray}\ (k,\ell)\in \mathbb N^{n'}\times \mathbb N^{n} 
	  					\\ |k|+|\ell|\geq 2,\ell\not=0 \end{subarray}}  
							{G}_{i,k,\ell}\Big( \sum_{\begin{subarray}\
							\sum_{s=1}^n\big( \sum_{j=1}^{\ell_{s}} p^{s,j}\big)+k=p\\
							\sum_{s=1}^n\big( \sum_{j=1}^{\ell_{s}} q^{s,j}\big)+1=q
								\end{subarray}}
							\prod_{j=1}^{\ell_1}M_{1,p^{1,j},q^{1,j}}
						\prod_{j=1}^{\ell_2}M_{2,p^{2,j},q^{2,j}}	
			 \\ & 		\cdots \cdots
						\prod_{j=1}^{\ell_n}M_{n,p^{n,j},q^{n,j}} 
						\Big)\Big |\frac{|\xi|^{q-\gamma}}{\Gamma(q/\gamma)}e^{c|\xi|^{\gamma}}.
	  	\end{aligned}
		\end{equation}
			  Hence we define for $q\geq 2$
	  \begin{equation}
	  	\begin{aligned}
	  	M_{i,p,q}= & A M_{i,p,q-1} 
		\\ & +C
			\sum_{\begin{subarray}\ (k,\ell)\in \mathbb N^{n'}\times \mathbb N^{n} 
							\\ |k|+|\ell|\geq 2 ,\ell\not=0
							\end{subarray}}  
							{G}_{i,k,\ell}\big( \sum_{\begin{subarray}\
							\sum_{s=1}^n\big( \sum_{j=1}^{\ell_{s}} p^{s,j}\big)+k=p \\
							\sum_{s=1}^n\big( \sum_{j=1}^{\ell_{s}} q^{s,j}\big)+1=q	
								\end{subarray}}
							\prod_{j=1}^{\ell_1}M_{1,p^{1,j},q^{1,j}}
							\prod_{j=1}^{\ell_2}M_{2,p^{2,j},q^{2,j}}	
			 			\\ & \quad		\cdots \cdots 
							\prod_{j=1}^{\ell_n}M_{n,p^{n,j},q^{n,j}}\big ). 
					\end{aligned} \label{M}
	  \end{equation}	  
		 Thus we have 
	  \begin{equation}
	  	\begin{aligned}
	  	|C_{i,p,q}(\xi)|\leq \sum_{i=1}^2 D_i \leq 
	  	\frac{ M_{i,p,q}|\xi|^{q-\gamma}e^{c|\xi|^{\gamma}}}{\Gamma(q/\gamma)}.
	  	\end{aligned}
	  \end{equation} 
		 (2) {\it Convergence of $\sum_{|p|\geq 2,q\geq 1} M_{i,p,q}Z^p s^q$.}\;  
		    For this purpose we use the method of implicit functions used in \cite{GT} and others. 
			  We introduce holomorphic functions 
	  \begin{equation*}
	  	\begin{aligned}
	  		  	G_{i}(Z,U)&=\sum_{\begin{subarray}\ (k, \ell)\in {\mathbb N}^{n'}\times
	  		  	 {\mathbb N}^n, \\
	  		  								|k|+|\ell|\geq 2, \ell\not=0  \end{subarray}}G_{i,k,\ell}Z^k U^{\ell}, \quad	
	  		  	H_i(Z) =\sum_{|p|\geq 2} H_{i,p} Z^p.								
	  	\end{aligned}
	  \end{equation*}
	  at 
	  $(Z,U)=(0,0)\in {\mathbb C}^{n'}\times {\mathbb C}^n$, $Z=(z_1,\cdots,z_{n'})$.
	  $U=(u_1,\cdots,u_n)$.
			Let
	  \begin{equation}
	  	F_{i}(Z,s,U)=sAu_i+sCG_{i}(Z,U)+sCH_i(Z) \label{F_i}
	  \end{equation}
	   and consider a system of functional equations with unknown functions
		  $U=(u_1,u_2,\cdots,u_n)$
	  \begin{equation}
	  	\begin{aligned}
		  	u_i  =F_{i}(Z,s,U)  \quad i=1,\cdots,n.
	   	\end{aligned}\label{Eq-u}
	  \end{equation}
		  We have 
	  \begin{lem} There exists a unique solution   
			  $U(Z,s)=(u_1(Z,s), \cdots,u_n(Z,s))$ of \eqref{Eq-u} such that it is holomorphic at
			  $(Z,s)=(0,0)$ and $U(0,0)=0$ with expansion  
		  \begin{equation}
			   u_{i}(Z,s)=\sum_{|p|\geq 2,q\geq 1}u_{i,p,q} Z^p s^q. 
	 	 \end{equation}
	 		 Moreover $M_{i,p,q}=u_{i,p,q}$. \label{lem-2}
	 	\end{lem}
	 		  Proposition \ref{prop-M} follows from Lemma \ref{lem-2}. \\
			{\it Proof of Lemma \ref{lem-2}.} 
		 		It follows from $F_{i}(0,0,0)=0$ and $(\frac{\partial F_i}{\partial u_j})=0$ at $(Z,s,U)=(0,0,0)$
		 		that there exists a unique solution
		 		 $U(Z,s)=(u_1(Z,s),\cdots,u_n(Z,s))$ of \eqref{Eq-u} with $u_i(0,0)=0$ for $1\leq i\leq n$.
		 	 Let  $	u_{i}(Z,s)=\sum_{|p|+q\geq 1}u_{i,p,q} Z^p s^q$. We have from \eqref{F_i} and
	 	 	\eqref{Eq-u}   
	 \begin{equation*}
		 \begin{aligned}
		 		{u}_{i,p,q}&=A{u}_{i,p,q-1}+\\ &  +C
				\sum_{\begin{subarray}\ (k,\ell)\in \mathbb N^{n'}\times \mathbb N^{n} 
							\\ |k|+|\ell|\geq 2, \ell\not=0 \end{subarray}}  
							{G}_{i,k,\ell}\big( \sum_{\begin{subarray}\
							\sum_{s=1}^n( \sum_{j=1}^{\ell_{s}} p_{s,j})+k=p \\
							\sum_{s=1}^n( \sum_{j=1}^{\ell_{s}} q_{s,j})+1=q	
								\end{subarray}}
			   		\prod_{j=1}^{\ell_1}{u}_{1,p_{1,j},q_{1,j}}
						\prod_{j=1}^{\ell_2}{u}_{2,p_{2,j},q_{2,j}}\cdots \cdots \\ \qquad &
						\cdots 
						\prod_{j=1}^{\ell_n}{u}_{n,p_{n,j},q_{n,j}} 
						\big) + \delta_{q,1}C{H}_{i,p}
		\end{aligned}
	\end{equation*}
	    Since $\delta_{q,1}{H}_{i,p}=0$ for $q\not =1$, $u_{i,p,0}=0$, We have $u_{i,p,1}=CH_p=M_{i,p,1}$ for 
	    $|p|\geq 2$ and $u_{i,p,1}=0$ for $|p|\leq 1$.  
		  Assume $u_{i,p,q'}=M_{i,p,q'}$ for $q'<q$. Then by \eqref{M}  $u_{i,p,q}=M_{i,p,q}$.
  	  \qed 
	\subsection{\normalsize Construction of $\Phi(x,z)$-III }
				Let us show $\sum_{|p|\geq 2, q\geq 1}	\widehat{C}_{i,p,q}(\xi)Z^{p}{\varepsilon}^q$ is convergent. 
	  	  It follows from Proposition \ref{prop-M} that there exist $A,B,C$ and $c'_{\varepsilon}>0$ such that 
	  	   	$M_{i,p,q}\leq A^{|p|}B^q$ and  
	   \begin{equation*}
	   	\begin{aligned}\ &
	   	\sum_{|p|\geq 2, q\geq 1}	|\widehat{C}_{i,p,q}(\xi)Z^{p}{\varepsilon}^q|\leq 
	   	 \sum_{|p|\geq 2, q\geq 1}\frac{M_{i,p,q}|\xi|^{q-\gamma}|Z^{p}{\varepsilon}^q|}
	   	 {\Gamma(q/\gamma)}e^{c|\xi|^{\gamma}}	\\  &\leq 
	   	  \sum_{|p|\geq 2, q\geq 1}\frac{A^{|p|}B^q|\xi|^{q-\gamma}|Z^{p}{\varepsilon}^q|}
	   	 {\Gamma(q/\gamma)}e^{c|\xi|^{\gamma}} 
	   	 \leq C(
	   	  \sum_{|p|\geq 2}{A^{|p|}|Z^{p}|})
	   	  \frac{|\xi|^{1-\gamma}e^{(c+c'_{\varepsilon})|\xi|^{\gamma}}}{\Gamma(1/\gamma)},
	   	  \end{aligned}
	   \end{equation*}	
		 		which converges in $\{Z \in {\mathbb C}^{n'}; |z|<A\}$, and
			  $\{\widehat{C}_{i,p,q}(\xi)\}\ (1\leq i \leq n)$ satisfy \eqref{Eq-c1}.
		   Take $\varepsilon=1$ and let 
	   \begin{equation}
		   	\widehat{C}_{i,p}(\xi)=\sum_{q\geq 1}\widehat{C}_{i,p,q}(\xi),		   	
	   \end{equation} 
	  	 and define  
		\begin{equation}
				C_{i,p}(x)=\int_{L(\theta)}e^{-(\frac{\xi}{x})^{\gamma}}
				\widehat{C}_{i,p}(\xi)d\xi^{\gamma} \quad \theta \in \widehat{I}_{\epsilon}.
				\end{equation}
	 Then
	  $ 	u_i(x,Z) =\sum_{|p|\geq 2}{C}_{i,p}(x)Z^{p}$. 
	     $\{u_i(x,Z)\}_{i=1}^n$ satisfy \eqref{Eq-U''}, hence $U(x,Z)=(u_1(x,Z), \cdots,u_n(x,Z))$ is
	   a solution of \eqref{Eq-U0}. 
	   Consequently $\Phi(x,Z)=U(x,Z)+\Psi(Z)=(\phi_1(x,Z),\cdots,\phi_{n}(x,Z))$ 
	   is a solution of 
	   equation \eqref{Eq-Phi} such that $\phi_i(x,Z)\in {\mathscr O}_{\{1/\gamma\}}
   		(S_{0},I_{\epsilon})\times \{|z|<r_{\epsilon} \})$
   		with $\phi_{i}(x,Z)=z_i+O(|Z|^2)$ for $1\leq i\leq n'$ and $\phi_{i}(x,z)=O(|Z|^2)$ for 
   		$i > n'$. Thus we get Theorem \ref{main-th}.
 	\subsection{\normalsize Remarks}  
			  We can take $\theta_{\mathbf \Lambda'}=\widehat{\theta}_0$  and 
			 	 $\theta_*=\theta_{\mathbf \Lambda'}/\gamma$ (see Remark \ref{rem2.3}). 
			 	 By transformation $x=t \exp(i\theta_{*})$, 
		 	 $\frac{d}{dx}=\exp({-i\theta_{*}})\frac{d}{dt}$ and 
		   $x^{1+\gamma}\frac{d}{dx}=e^{i\gamma\theta_{*}}t^{1+\gamma}\frac{d}{dt}$,
	  \begin{equation*}
	  	\begin{aligned}
	  		x^{1+\gamma}\frac{d}{dx}-\lambda_i
	  		=e^{i\theta_{\mathbf \Lambda'}}(t^{1+\gamma}\frac{d}{dt}
	  		-e^{-i\theta_{\mathbf \Lambda'}}\lambda_i).
	  	\end{aligned}
	  \end{equation*}	  
			 	 We have $\arg(e^{-i\theta_{\mathbf \Lambda'}}\lambda_i)=
			  \omega_i-\theta_{\mathbf \Lambda'}$. Hence we may assume $\theta_{\mathbf \Lambda'}=0$ and  	  
			  $$ \arg \lambda_i=\omega_i,\; |\omega_i|<\delta<\pi/2, 1\leq i \leq n'.$$    
	  		Then there exist $\delta_*>\pi/2\gamma$ and $r=r(\delta_*)>0$ such that   
	  \begin{equation}
	  	\begin{aligned}
	  		\phi_i(x,Z)\in {\mathscr O}_{\{1/\gamma\}}\big(S_{0}(I)\times \{|Z|<r \}\big), \;	 
	  		I=(-\delta_*,\delta_*).	 
	  	\end{aligned}
	  \end{equation}
				Let 
		\begin{equation}
			\begin{aligned}
				z_{i}(x)=e^{h_{i}(x)} \quad  h_{i}(x)=\int^{x} \frac{\lambda_{i}(t)}{t^{\gamma+1}}dt
				=-\frac{\lambda_i}{\gamma x^{\gamma}}+\cdots, \;\;1\leq i \leq n' 
			\end{aligned}
		\end{equation}
			and $\varepsilon>0$ be a small constant. If $|\gamma\arg x-\omega_{i}|<(\pi/2-\varepsilon)$, 
			$z_{i}(x)$ exponetially decreases with order $\gamma$. Let ${\mathcal I(\varepsilon)}
			=\cap_{i=1}^{n'}\{ \theta; |\gamma\theta-\omega_{i}|<(\pi/2-\varepsilon)\}$.
	 		Then $\{ \theta; |\gamma\theta|<(\pi/2-\delta-\varepsilon)\}\subset {\mathcal I(\varepsilon)}$.
	 		We have from Theorem \ref{main-th}
	\begin{cor}
				Let $x\in S_{0}({\mathcal I(\varepsilon)})$ and $y_i(x)=\phi_{i}(x,C_1z_1(x),
			 \cdots, C_{n'}z_{n'}(x))$, where $|C_kz_k(x)|<r$. 
			Then  $Y(x)={}^t(y_1(x),\cdots,y_n(x))$ is a solution of \eqref{Eq-Y}.
	\end{cor}
\section{\normalsize Nonlinear equation with irregular singularity II} 
		In this section we study
	\begin{equation}\left \{
		\begin{aligned}
			Y=&{}^t(y_1.y_2,\cdots,y_n) \\ 	
				x^{1+\gamma}\frac{dY}{dx}=& F_0(x)+A(x)Y+F(x,Y),
  	\end{aligned} \right . \label{Eq-01}
	\end{equation}
			where $A(x)=(a_{i,j}(x))_{1\leq i,j \leq n}$, $F_0(x)\;$ and $a_{i,j}(x)$ are holomorphic in 
			a neighborhood of $x=0$ and $ F_0(0)=0$. 
			$F(x,Y)$ is holomorphic in a neighborhood of $(x,Y)=(0,0)$ and $F(x,Y)=O(|Y|^2)$.	
			Let $\{\lambda_i\}_{1\leq i \leq n}$ be eigenvalues of $A(0)$ and 
			$\lambda_i\not =0$ for all $ i$. Let $\omega_i=\arg \lambda_i  \;(0\leq \omega_i<2\pi)$ and
	\begin{equation}
			{\mathbf	\Theta}_0=\{ (\omega_i+2\pi\ell)/\gamma,\; 1\leq i \leq n,
			 \ell \in \mathbb{Z}\}.
	\end{equation}
			There exists a unique formal solution 
			$\widetilde{K}(x)\in {\mathbb C}[[x]]^{n}$ with $\widetilde{K}(0)=0$ of \eqref{Eq-01}.		 
			Its Borel summability follows from \cite{Braak}. We have
	\begin{prop} Suppose ${\theta}_*\not \in {\mathbf	\Theta}_0$ 
		 	Then there exists a solution $K(x)$ of \eqref{Eq-01},
		 	 which is $\gamma$-Borel summable in the direction ${\theta}_*$
		 	with the asymptotic expansion $\widetilde{K}(x)$.
	\end{prop}	
			Put $Y=xW+K(x)$. Then
		\begin{equation*}
			\begin{aligned}
				x^{1+\gamma}\frac{dW}{dx}=&(A(x)-x^{\gamma}I)W+x^{-1}\big
				(F(x,xW+K(x))-F(x,K(x))\big)\\ 
				x^{1+\gamma}\frac{dw_i}{dx}=& \sum_{j=1}^n (a_{i,j}(x)-\delta_{i,j}x^{\gamma})w_j+x^{-1}
  			\big(f_i(x,xW+K(x))-f_i(x,K(x))\big)\\ 
  			=& \sum_{j=1}^n (a_{i,j}(x)-\delta_{i,j}x^{\gamma})w_j 
				+\sum_{j=1}^{n}\frac{\partial }{\partial y_{j}}f_i(x,K(x))w_j+g_i(x,W).
			\end{aligned} 
		\end{equation*}	
			Put $A'(x)=\big(a_{i,j}(x)-\delta_{i,j}x^{\gamma}
			+\frac{\partial }{\partial y_{j}}f_i(x,K(x))\big)$. Then
		\begin{equation}\left \{
 			\begin{aligned}
			W=&{}^t(w_1.w_2,\cdots,w_n) \\ 	
				x^{1+\gamma}\frac{dW}{dx}=& A'(x)W+G(x,W),
				\end{aligned} \right . \label{Eq-02}
		\end{equation}
			where $A'(0)=A(0)$ and
			$G(x,W)={}^t(g_1(x,W), \cdots,g_n(x,W))$ with $G(0,W)=0$ and $G(x,W)=O(|W|^2)$.
			 Assume $\{\lambda_i\}_{1\leq i \leq n}$  are distinct (Condition 0) and 
			${\theta}_*\not \in {\mathbf	\Theta}_0\cup {\mathbf	\Theta}_1$. 
			Then by an invertible linear transformation $W=P(x)U$ whose
			elements in ${\mathscr O}_{\{1/\gamma\}}(S_0(I))$ 
			($I=({\theta}_*-\delta_*,,{\theta}_*+\delta_*)$, $\delta_*>\pi/2\gamma $), we have 
		\begin{equation}\left \{	
			\begin{aligned}
				U=&{}^t(u_1.u_2,\cdots,u_n) \\ 	
				x^{1+\gamma}\frac{dU}{dx}=& B(x)U+H(x,U)\\ 
				B(x)=& diag.\;(b_1(x),b_2(x),\cdots,b_n(x)) \\
				H(x,U)=&{}^t(h_1(x,U),h_2(x,U),\cdots,h_n(x,U)),
			\end{aligned} \right . \label{Eq-03}
		\end{equation}
				where $b_i(x)$ is a polynomial with degree $\leq \gamma$ and $b_i(0)=\lambda_i$
				and $h_i(x,U)\in {\mathscr O}_{\{1/\gamma\}}(S_0(I)\times \Omega)$, $\Omega=\{U \in \mathbb{C}^n;
				|U|<R\}$, with $h_i(0,U)=0$ and
				$h_i(x,U)=O(|U|^2)$. Consequently we get \eqref{Eq-03} 
				from \eqref{Eq-01} by a transform $Y=K(x)+xW=K(x)+xP(x)U$. 
				We remark that $B(x)$ depends on $K(x)$. \par
				Set $\Lambda'=\{\lambda_i;1\leq i \leq n'\}$ and assume Conditions 1 and 2. 
				Take $\theta_*$ so that it satisfies the assumption of Theorem \ref{main-th}. Consider 
				an $n'\times n'$  system of linear equations
		\begin{equation}\left \{	
			\begin{aligned}
						Z=&{}^t(z_1.z_2,\cdots,z_{n'}) \\ 	
						x^{1+\gamma}\frac{dz_i}{dx}=& b_i(x)z_i, \; 1\leq i \leq n'.
			\end{aligned} \right . \label{Eq-04}
		\end{equation}
				By applying Theorem \ref{main-th}, we have
		\begin{thm} 
	 		   There exists $\Phi(x,Z)=(\phi_1(x,Z),\cdots,\phi_{n}(x,Z))$
	   		 such that for any small $\epsilon>0$ there exists $r_{\epsilon}>0$,
	   		$\phi_i(x,Z)\in {\mathscr O}_{\{1/\gamma\}}
	   		(S_{\{0\}}(I_{\epsilon})\times \{Z\in {\mathbb C}^{n'}; |Z|<r_{\epsilon} \})$
	   		and the followings hold.
	   	\begin{enumerate}
			   	 \item[{\rm (1)}] $\phi_{i}(x,Z)=z_i+O(|Z|^2)$ for $1\leq i\leq n'$ and 
	   				$\phi_{i}(x,Z)=O(|Z|^2)$ for $i > n'$. 
		   		 \item[{\rm (2)}] Let $\mathcal S$  be an open set in $S_{\{0\}}(I_{\epsilon})$ and 
   					$Z(x)=(z_1(x),\cdots,z_{n'}(x))\; (x\in \mathcal S,|Z(x)|<r_{\epsilon})$ be a solution of
   					\eqref{Eq-04}.
							Then $Y(x)=K(x)+xP(x)\Phi(x,Z(x))$ satisfies \eqref{Eq-01} in $\mathcal S$.
   		 \end{enumerate} \label{main-1}
  	\end{thm}
 	 \section{\normalsize Applications }
 	 		\subsection{\normalsize A special $2\times 2$ system}  
 	 		Let us consider a special $2\times 2$ system as an example, 
 	 			\begin{equation}
 	 				\begin{aligned}
					x^{1+\gamma}\frac{dY}{dx}= F_0(x)+A(x)Y+F(x,Y) \quad Y= 
						\begin{bmatrix}	
							y_1\\
							y_2				
						\end{bmatrix}.  
					\end{aligned} \label{2,2system}
				\end{equation}
			The assumptions for \eqref{2,2system} are the same as that for \eqref{Eq-01}.  
 	 		Let $\lambda_1, \lambda_2$ be eigenvalues of $A(0)$. Further we assume 
 	 		\begin{itembox}[l]{\bf Condition 5-1} 
  				 $\lambda_1\lambda_2 \not =0$ and	$\arg\lambda_1=\omega$, $\arg\lambda_2=\omega+\pi$. 
			\end{itembox}
  	 		Then 
 	 		$\Theta_0\cup\Theta_1=\{(\omega+\ell\pi)/\gamma; \ell \in \mathbb{Z} \}$. If $\Lambda'=\{\lambda_1\}$
 	 		 ($\Lambda'=\{\lambda_2\}$), then $\frak{L}\subset L(\omega)$ (resp. $\frak{L}\subset L(\omega+\pi))$
 	 		(see $\eqref{fL}$).
 	 		Let $\theta_* \not \in \Theta_0\cup\Theta_1$ and $K(x)={}^t(k_1(x),k_2(x))$ be a solution of 
 	 		\eqref{2,2system} with $\gamma$-Borel summable in the direction $\theta_*$.  
 	 		By eliminating 
 	 		$F_0(x)$ and a linear transformation of the unknowns $Y=K(x)+xP(x)U$, 
 	 		$U={}^t(u_1,u_2)$, we can reduce \eqref{2,2system} to             
 	 		\begin{equation}
 			\begin{aligned}
 			   	x^{1+\gamma}\frac{d}{dx}
 			   	\begin{bmatrix}
 			   	  u_1 \\
 			   	  u_2 
 			   	\end{bmatrix}=
 			   	\begin{bmatrix}
 			   	 \lambda_1(x) & 0 \\
 			   	  0 & \lambda_2(x) 
 			   	\end{bmatrix}
 			   		\begin{bmatrix}
 			   	  u_1 \\
 			   	  u_2 
 			   	\end{bmatrix}+ G(x,U), \; G(x,U)=
 			   	\begin{bmatrix}
 			   	  g_1(x,U) \\
 			   	  g_2(x,U)  
 			   	\end{bmatrix} 
 			   	\end{aligned}.  \label{Eq-U}
		\end{equation}
		$\lambda_i(x)\; (i=1,2)$ is a polynomial with degree $\leq \gamma$ and
		$\lambda_i(0)=\lambda_i$.  
		$G(x,U)$ is $\gamma$-Borel summable in the direction $\theta_*$ with respect to $x$ 
		such that $G(x,U)=O(|U|^2)$ and $G(0,U)=0$. \par
		 Let $A_i(x)=\int^{x}\frac{\lambda_i(\tau)}{\tau^{\gamma+1}}d\tau$.  
		  Then $Z={}^t(C_1e^{A_1(x)} ,C_2e^{A_2(x)} )$ is a solution of
		 \begin{equation}
 			\begin{aligned}
 			   	x^{1+\gamma}\frac{d}{dx}
 			   	\begin{bmatrix}
 			   	  z_1 \\
 			   	  z_2 
 			   	\end{bmatrix}=
 			   	\begin{bmatrix}
 			   	 \lambda_1(x) & 0 \\
 			   	  0 & \lambda_2(x) 
 			   	\end{bmatrix}
 			   		\begin{bmatrix}
 			   	  z_1 \\
 			   	  z_2 
 			   	\end{bmatrix}.
 			   \end{aligned}
 		\end{equation}
 				Take $\theta_*\in (\frac{\omega}{\gamma},\frac{\omega+\pi}{\gamma})$ and let  
			$I=(\frac{\omega}{\gamma}-\frac{\pi}{2\gamma},\frac{\omega}{\gamma}+\frac{3\pi}{2\gamma})$.	 
			Then the assumptions of Theorem \ref{main-1} hold for both 
			$\Lambda'=\{\lambda_1\}$ and $\Lambda'=\{\lambda_2\}$. Hence 	   	
			 there exist $\Psi_1(x,z_1)={}^t(\psi_{1,1}(x,z_1),\psi_{1,2}(x,z_1))$
				and  $\Psi_2(x,z_2)={}^t(\psi_{2,1}(x,z_2),\psi_{2,2}(x,z_2))$ with  
			$\psi_{i,j}(x,z_i)=\sum_{n=1}^{\infty}\psi_{i,j}^n(x)z_i^n  \in \mathscr{O}_{\{1/\gamma\}}
			( S_{\{0\}}(I_{\epsilon})\times \{|z_i|<r_{\epsilon}\})$ for any small $\epsilon>0$ and 
			$|\partial_{z_i}\psi_{i,1}(x,0)|+|\partial_{z_i}\psi_{i,2}(x,0)|\not \equiv 0$ such that they 
		  have the following properties: 
	 \begin{thm} \label{2*2}
	     Let $\epsilon>0$ be an arbitrary small constant. \\
			{\rm (1)} Let ${\mathcal I}^{\frac{\omega}{\gamma}}=(\frac{\omega}{\gamma}-\frac{\pi}{2\gamma}, 
				\frac{\omega}{\gamma}+\frac{\pi}{2\gamma})$.  
			 Then there exists $C({\epsilon})>0$ such that $Y(x)={}^t(y_1(x),y_2(x))$,
		\begin{equation} 
				\begin{aligned}
					y_j(x)=k_j(x)+\sum_{n=1}^{\infty}\psi_{1,j}^n(x)C_1^ne^{nA_1(x)}\quad  j=1,2,
				\end{aligned}
		\end{equation}
				with $|C_1|<C({\epsilon})$, is a solution of \eqref{2,2system} in 
				$S_{\{0\}}({\mathcal I}^{\frac{\omega}{\gamma}}_{\epsilon})$. 
		\\
		{\rm (2)}	Let ${\mathcal I}^{\frac{\omega+\pi}{\gamma}}=
		(\frac{\omega}{\gamma}+\frac{\pi}{2\gamma}, \frac{\omega}{\gamma}+\frac{3\pi}{2\gamma})$.
		Then there exists $C({\epsilon})>0$ such that $Y(x)={}^t(y_1(x),y_2(x))$,  
		\begin{equation}
			\begin{aligned}
			y_j(x)=k_j(x)+\sum_{n=1}^{\infty}\psi_{2,j}^n(x)C_2^ne^{nA_2(x)}\quad  j=1,2,
			\end{aligned}
		\end{equation}
		with $|C_2|<C({\epsilon})$, is a solution of \eqref{2,2system} in 
		$S_{\{0\}}({\mathcal I}^{\frac{\omega+\pi}{\gamma}}_{\epsilon})$.
		 \end{thm}
		 \begin{proof}
		 	Since $e^{A_1(x)}\; (e^{A_2x}) $ decays exponentially in 
		$S_{\{0\}}({\mathcal I}^{\omega/\gamma}_{\epsilon})$ 
		 $(resp. S_{\{0\}}({\mathcal I}^{(\omega+\pi)/\gamma}_{\epsilon})$). 
		$Y(x)={}^t(y_1(x),y_2(x))$ is a solution of \eqref{Eq-01} in $S_{\{0\}}({\mathcal I}^{\omega/\gamma}_{\epsilon})$ 
		 $(resp. S_{\{0\}}({\mathcal I}^{(\omega+\pi)/\gamma}_{\epsilon})$).\\
		 \end{proof}
	\subsection{\normalsize Painlev\'{e} 2 and Painlev\'{e} 4}  
			We try to apply Theorem \ref{2*2} to Painlev\'{e} 2 and 4 equations as examples. 
			Other Painlev\'{e} equations at irregular singular points are also studied
			in the same way.  
			Painlev\'{e} 5, where $\gamma=1$, is studied in \cite{Cos-3} in a different method. 
			The present cases have $\gamma>1$. First let us study Painlev\'{e} 2, 	 
		\begin{equation}
			\begin{aligned}
				(P)_2:\quad y''=2y^3+ty+a.
			\end{aligned} 
		\end{equation}
			By $t=1/s$
		\begin{equation}
			\begin{aligned}
				\quad s(s^2\frac{d}{ds})^2y=2sy^3+y+as.
			\end{aligned} \label{P2}
		\end{equation}
			($P_2.1$)	 There exists a unique formal power series solution 
			$\widetilde{k}(s)=-as+2(a^3-a)s^4+\cdots$ to \eqref{P2}. Let $y=-as+z$. Then we have
		\begin{equation*}
			\begin{aligned}\ &
			 s(s^2\frac{d}{ds})^2z=(1+6a^2s^3)z+2(a-a^3)s^4-6as^2z^2+2sz^3,
			\end{aligned}
		\end{equation*}	and
		\begin{equation*}
			\begin{aligned}& 
				(s^{5/2}\frac{d}{ds})^2z-\frac{1}{2}s^4\frac{d}{ds}z 
				=(1+6a^2s^3)z+a_0(s)+a_2(s)z^2+a_3(s)z^3.
			\end{aligned}
		\end{equation*}
		   Set $u=z$, $v=s^{5/2}\frac{d}{ds}z$. Then
		\begin{equation}
    	\begin{aligned}\
    	s^{5/2} \frac{d}{ds}
    	 \begin{bmatrix}
  	   u \\
  	   v 
  		\end{bmatrix} = &
 		\left (
  		 	 \begin{bmatrix}
  	  0 & 1 \\
  	  1 & 0 
  		\end{bmatrix} 				 
				+s^{3/2}
  				 	 \begin{bmatrix}
  	    0  & 0  \\ 
	 	    0  & \frac{1}{2}    
  		\end{bmatrix} 
  		+O(s^3)		 
 			\right )
      \begin{bmatrix}
  	   u \\
  	   v 
  		\end{bmatrix} \\ &+
  	 	 \begin{bmatrix}
  		0 \\
  	a_0(s)+a_2(s)u^2+a_3(s)u^3 
  		\end{bmatrix} 	\end{aligned}.		
  \end{equation}
		By changing $x=s^{1/2}$
		\begin{equation}
    	\begin{aligned}\
    	x^{4} \frac{d}{dx}
    	 \begin{bmatrix}
  	   u \\
  	   v 
  		\end{bmatrix} = &
 		\left (
  		 	 \begin{bmatrix}
  	  0 & 2 \\
  	  2 & 0 
  		\end{bmatrix} 				 
				+x^{3}
  				 	 \begin{bmatrix}
  	    0  & 0  \\ 
  	    0  & 1  
  		\end{bmatrix} 
  		+O(x^6)		 
 			\right )
      \begin{bmatrix}
  	   u \\
  	   v 
  		\end{bmatrix} \\ &+2
  	 	 \begin{bmatrix}
  		0 \\
  	a_0(x^2)+a_2(x^2)u^2+a_3(x^2)u^3 
  		\end{bmatrix} 	\end{aligned}.\label{P2-1}.		
  \end{equation}
			Then $\gamma=3$,$\lambda_1=2, \lambda_2=-2$ and ${\mathbf \Theta_0}\cup {\mathbf \Theta_1}=
			\{\frac{\pi\ell}{3}; \ell \in \mathbb{Z} \}$. Hence  equation \eqref{P2-1} 
			satisfies Condition 5.1 and we can apply Theorem \ref{2*2} to it. We also 
			conclude that  $\tilde{k}(s)\in {\mathbb C}[[s]]$ is $3/2$-Borel summable. \\
			($P_2.2$) Let 
			$2c^2+1=0$ and $y= s^{-1/2}(c+z)$. Then 
		\begin{equation*}
			\begin{aligned}\ &
		    (s^{5/2}\frac{d}{ds})^2 z-\frac{3}{2}{s^{4}}\frac{d}{ds}z
     =(6c^2+1+\frac{s^{3}}{4})z+6cz^2+2z^3+as^{3/2}+\frac{cs^{3}}{4}.
			\end{aligned}
		\end{equation*}
				   Set $u=z$, $v=s^{5/2}\frac{d}{ds}z$. Then
		\begin{equation}
    	\begin{aligned}\
    	s^{5/2} \frac{d}{ds}
    	 \begin{bmatrix}
  	   u \\
  	   v 
  		\end{bmatrix} = &
 		\left (
  		 	 \begin{bmatrix}
  	  0 & 1 \\
  	 -2 & 0 
  		\end{bmatrix} 				 
				+s^{3/2}
  				 	 \begin{bmatrix}
  	    0  & 0  \\ 
  	    0  & \frac{3}{2}    
  		\end{bmatrix} 
  		+O(s^3)		 
 			\right )
      \begin{bmatrix}
  	   u \\
  	   v 
  		\end{bmatrix} \\ &+
  	 	 \begin{bmatrix}
  		0 \\
  	a_0(s)+6cu^2+2u^3 
  		\end{bmatrix} 	\end{aligned}.		
  \end{equation}
   By changing $x=s^{1/2}$, we can show that Theorem \ref{2*2} is available to this case 
   in the same way as the case (1).
   \par 
   Next let us study Painlev\'{e}-4. 
		\begin{equation}
  	\begin{aligned}  (P)_4: \quad
  	y''=& 	\frac{(y')^2}{2y}+\frac{3}{2}y^3+4ty^2+2({t^2}-\alpha)y+\frac{\beta}{y}.
  		  \end{aligned}
  \end{equation}
  Let $t=1/x$. Then
   \begin{equation}
	   	\begin{aligned}
	  		x^2(x^2\frac{d}{dx})^2y=\frac{1}{2y}(x^3\frac{dy}{dx})^2+\frac{3}{2}x^2y^3+4x y^2
  			+2(1-\alpha x^2)y+\frac{\beta x^2}{y}..
    	\end{aligned} \label{5.12}
  \end{equation}
 	 	$(P_4.1)$\; Futher by multiplying \eqref{5.12} by $x$, we have    
     \begin{equation*}
	   	\begin{aligned}
	  		x^3(x^2\frac{d}{dx})^2y=\frac{x}{2y}(x^3\frac{dy}{dx})^2+\frac{3}{2}x^3y^3+4x^2y^2
  			+2(1-\alpha x^2)xy+\frac{\beta }{y} x^3.
    	\end{aligned}
  \end{equation*}
 		Let $3c^2+8c+4=0$\; $(c=-\frac{2}{3},-2)$.
    By transformation $y=\frac{1}{x}(c+z)$, we obtain  
    \begin{equation*}
  	\begin{aligned}\ &
  	x^3(x^2\frac{d}{dx})^2y=(x^3\frac{d}{dx})^2z-3x^5\frac{d}{dx}z
  	 \\ =&\frac{x^2}{2c}(1+\frac{z}{c})^{-1}(x^3\frac{d}{dx}\frac{c+z}{x})^2
  	+\frac{3}{2}(c+z)^3 +4(c+z)^2\\
  	&+2(c+z)	-2\alpha x^2(c+z) 
  	 +\frac{\beta x^4 }{c}(1+\frac{z}{c})^{-1},
  	\end{aligned}
  \end{equation*}
  By $(c+z)^{-1}=c^{-1} (\sum_{n=0}^{\infty}(-\frac{z}{c})^n)$ we get
  \begin{equation}
  (x^3\frac{d}{dx})^2z=2x^5z'+a(x)z+b(x,z,x^3\frac{d}{dx}z)+h(x),
  \end{equation}
  where 
    \begin{equation*}
  	\begin{aligned}
  	a(x)=&  \frac{9}{2}c^2+8c+2-2\alpha x^2 +O(x^4),\; 
   b(x,z,p)=\sum_{i+j\geq 2}b_{i,j}(x)z^i p^j \\
   h(x)=&\frac{3}{2}c^3+4c^2 +2c-2\alpha c x^2+O(x^4)=-2\alpha c x^2+O(x^4).
  	\end{aligned}
  \end{equation*}
   It follows from $3c^2+8c +4=0$ that  $h(x)=O(x^2)$ and 
   	$a(x)=-4(c+1)+O(x^2)$ \; ($-4(c+1)= -4/3$ or $4$). 
		Let $u=z, v=x^3z'$. Then we have 
\begin{equation}
    	\begin{aligned}\
    	x^3 \frac{d}{dx}
    	 \begin{bmatrix}
  	   u \\
  	   v 
  		\end{bmatrix} = &
 		\left (
  		 	 \begin{bmatrix}
  	  0 & 1 \\
  	 -4(c+1) & 0 
  		\end{bmatrix} 				 
				+{x^2}
  				 	 \begin{bmatrix}
  	  0  & 0    \\ 
  	  -2\alpha   & 2  	      
  		\end{bmatrix} 
 			+O(x^4)\right )
      \begin{bmatrix}
  	   u \\
  	   v 
  		\end{bmatrix} \\ &+
  	 	 \begin{bmatrix}
  		0 \\
  	b(x,u,v)+h(x) 
  		\end{bmatrix}. 	\end{aligned}.		
  \end{equation}
  and $\gamma=2$. Hence we can apply Theorem \ref{2*2} to this case. \\
		$(P_4.2)$\; 
		 Let $2c^2+\beta=0 \; (\beta \not=0)$ and $y=x(c+z)$. Then from \eqref{5.12}
  \begin{equation*}\left \{
  	\begin{aligned}\  &
  	x^2(x^2\frac{d}{dx})^2x(c+z)=x(x^3\frac{d}{dx})^2z+x^6\frac{dz}{dx}+2x^5z+2cx^5
  	\\ &		\frac{3}{2}x^2y^3+4xy^2	-2\alpha x^2y
			=	x^3(b_0+b_1z+b_2(x)z^2+b_3(x)z^3)  
			\\ & b_1=\frac{9c^2}{2}x^2+8c-2\alpha 		\\ &
			2y+\frac{\beta x^2}{y} =4xz+ \frac{\beta x}{c}\sum_{n=2}^{\infty}(-z/c)^{n}\; 
			(\because \beta+2c^2=0)			
  			  	\end{aligned}\right.
  \end{equation*}
  and 
  \begin{equation*}
  	\begin{aligned}\ &
  	 	\frac{1}{2y}(x^3\frac{dy}{dx})^2= x^5(\frac{d(xz)}{dx}+c)^2
  (\frac{1}{2c}\sum_{n=0}^{\infty}(-z/c)^{n}) \\ 
  = & \frac{cx^5}{2}+(O(x^5))z+
  x^6\frac{dz}{dx}+ x \sum_{i+j\geq 2}b_{i,j}(x)z^i(x^3\frac{dz}{dx})^j 
  	\end{aligned}
  \end{equation*}
  Thus we get
  \begin{equation}
  	\begin{aligned}\  &
  		(x^3\frac{d}{dx})^2z   
  =	 x^2h(x)+(4+(8c-2\alpha)x^2+O(x^4))z 
    + \big(\sum_{i+j\geq 2}b_{i,j}(x)z^i(x^3\frac{dz}{dx})^j \big ).
  			  	\end{aligned}
  \end{equation}
  		Let $u=z, v=x^3z'$. Then we have 
\begin{equation}
    	\begin{aligned}\
    	x^3 \frac{d}{dx}
    	 \begin{bmatrix}
  	   u \\
  	   v 
  		\end{bmatrix} = &
 		\left (
  		 	 \begin{bmatrix}
  	  0 & 1 \\
  	  4 & 0 
  		\end{bmatrix} 				 
				+{x^2}
  				 	 \begin{bmatrix}
  	  0  & 0    \\ 
  	  8c-2\alpha   & 0    
  		\end{bmatrix} 			 
 			+O(x^4)\right )
      \begin{bmatrix}
  	   u \\
  	   v 
  		\end{bmatrix} \\ &+
  	 	 \begin{bmatrix}
  		0 \\
  	c(x,u,v)+x^2h(x) 
  		\end{bmatrix}. 	\end{aligned}		
  \end{equation}
  Therefore we can apply Theorem \ref{2*2} to this case.
 \section {\normalsize 
				Transformation of systems of linear ordinary equations with irregular singularity}
				Let us return to an $n\times n$ system of linear ordinary equations 
		\begin{equation}
			x^{1+\gamma}\dfrac{dY}{dx}=A(x)Y. \; A(x)=(a_{i,j}(x))\label{Eq-0'}
		\end{equation}
				and prove Proposition \ref{prop-1}. Let $\lambda_{1},\lambda_2,\cdots, \lambda_n$ be the eigenvalues 
				of $A(0)$, $\lambda_{i}\not=\lambda_j \;(i\not= j)$ and $A(0)$ be diagonal.
				Let us recall the definitions $\omega_{i,k}$, $\theta_{i,k,\ell}$ 
				and $\mathbf \Theta_1$ (see \eqref{not-2}).
				 Let $\theta_* \not \in {\mathbf \Theta}_1$. Then 
				there exist $\widehat{I}=(\theta_*-\epsilon_* , \theta_* +\epsilon_*)$ $(\epsilon_*>0)$ 
				such that $\widehat{I} \cap {\mathbf \Theta_1} =\emptyset$. We assume 
				$\{a_{i,j}(x)\}_{1\leq i,j \leq n}$ are $\gamma$-Borel summable in the direction $\theta_*$. 
				Hence there exists $I=(\theta_*-\delta_*, \theta_* +\delta_*)$ $(\delta_*>\pi/2\gamma)$  
				such that 	$a_{i,j}(x) \in {\mathscr O}_{\{1/\gamma\}}(S_{0}(I))$. 
				Firstly we have, by a transformation with polynomial elements,
	\begin{lem}
				There is a matrix $P(x)$ $(P(0)=Id)$ with polynomial elements such that  
				linear transformation $Y=P(x)Z$ transforms \eqref{Eq-0'} to  
		\begin{equation}
		 		 x^{1+\gamma}\frac{dz_i}{dx}={\lambda}_i(x)z_i+\sum_{k=1}^{n}b_{i,k}(x)z_{k} \quad
			 (1\leq i \leq n),  \label{eq-0'}
		\end{equation}
				where  ${\lambda}_i(x)$ is a polynomial with degree 
				$\leq \gamma$, $\lambda_i(0)=\lambda_i$ and
				$b_{i,k}(x)(1\leq i,k\leq n)  \in {\mathscr O}_{\{1/\gamma\}}(S_{0}(I))$ 	with 
				$b_{i,k}(x)=O(x^{1+\gamma})$.
	\end{lem}
		 It is known that we can take $P(x)$ with formal power series elements so that
			$B(x)=(b_{i,j}(x))$ is also diagonal (See \cite{Was}). We only have to stop at finite steps. 
	\begin{proof}	
		Let $Y=P(x)Z$. Then 
	\begin{equation*}
		\begin{aligned}
	 		x^{1+\gamma}\frac{dY}{dx}=x^{1+\gamma}(\frac{dP}{dx}Z+P(x)\frac{dZ}{dx})
	 		=A(x)P(x)Z.
		\end{aligned}
	\end{equation*}
		Suppose $x^{1+\gamma}\dfrac{dZ}{dx}=B(x)Z$. Then
	\begin{equation*}
		  x^{1+\gamma}P'(x)+P(x)B(x)=A(x)P(x).
	\end{equation*}
		Let $A(x)=\sum_{k=0}^{\infty}A_k x^k$, $B(x)=\sum_{k=0}^{\infty}B_k x^k$
		and $P(x)=\sum_{k=0}^{\gamma}P_k x^k$. We have 
	\begin{equation*}
		\begin{aligned}
			\sum_{k+l=m}A_lP_k-\sum_{k+l=m}P_kB_l&=0 \quad m\leq \gamma,  \\
			\sum_{k+l=m}A_lP_k-\sum_{k+l=m}P_kB_l&=(m-\gamma)P_{m-\gamma}
			 \quad m\geq \gamma+1.
		\end{aligned}
	\end{equation*}
		Let $P_0=Id$ and $B_0=A_0$. For $1\leq m \leq \gamma$
	\begin{equation*}
		\begin{aligned}
		P_mB_0-A_0P_m
		=\sum_{k=1}^{m-1}(A_{m-k}P_k-P_kB_{m-k})+A_m-B_m. 
			\end{aligned}
	\end{equation*}
			Let $P_m=(p_{m,i,j})_{1\leq i,j\leq n}$ and $B_{m}=(b_{m,i,j})_{1\leq i,j\leq n}$.
			It follows from $A_0=B_0=diag.(\lambda_1,\cdots, \lambda_n)$ with $\lambda_i\not=
			\lambda_j (i\not=j)$ that we can take $p_{m,i,j}\ (i\not =j)$ so that 
			$b_{m,i,j}=0 \ (i\not =j,\ 0\leq m\leq \gamma)$. $P_m=0$ for $m\geq \gamma+1$.
	\end{proof}
		 	Therefore we may study 
		\begin{equation}
		\begin{aligned}
		x^{1+\gamma}\frac{dy_i}{dx}={\lambda}_i(x)y_i+\sum_{j=1
		}^{n}a_{i,j}(x)y_{j} \quad
		 (1\leq i \leq n)  \label{Eq-start}
		\end{aligned}
	\end{equation}
		with $a_{i,j}(x)=O(x^{1+\gamma})$. Set $n\times n$ matrices 
		$\Lambda(x)=diag.(\lambda_1(x), \cdots.\lambda_n(x))$	and $A(x)=(a_{i,j}(x))$.
		Our aim is to transform \eqref{Eq-start} to a simpler form, that is,
		to the following linear system of equations	
	\begin{equation}
		  x^{1+\gamma}\frac{dz_i}{dx}={\lambda}_i(x)z_i \quad
			 (1\leq i \leq n),
			  \label{Eq-simple}
	\end{equation}
   	 by using Borel summable functions. \par
	  {\it Proof of Proposition \ref{prop-1}}.\; 
		Let $Y=(Id+C(x))Z$, where $C(x)=(C_{i,k}(x))$ is an $n\times n$ matrix with $C(0)=0$. Then
	\begin{equation*}
			\begin{aligned}\ &
				 x^{1+\gamma}\frac{dY}{dx}=x^{1+\gamma}(Id+C(x))\frac{dZ}{dx}+x^{1+\gamma}\frac{dC}{dx}Z \\ &
				 =(Id+C(x))\Lambda(x)Z+x^{1+\gamma}\frac{dC}{dx}Z
				=({\Lambda}(x)+A(x))(Id+C(x))Z.
			\end{aligned}
	\end{equation*}	
			The equation to solve is  		
	\begin{equation}
			\begin{aligned}
				(Id+C(x))\Lambda(x)+x^{1+\gamma}\frac{dC}{dx}=({\Lambda}(x)+A(x))(Id+C(x)),
			\end{aligned}
	\end{equation}
		 more precisely 	
	\begin{equation*} 
		\begin{aligned} \ &
		 x^{1+\gamma}C_{i,k}'(x)+
		\big(\delta_{i,k}+C_{i,k}(x)\big){\lambda_k}(x)\\ & =
		\lambda_{i}(x)(\delta_{i,k}+C_{i,k}(x))+ 
		\sum_{j=1}^{n}a_{i,j}(x)(\delta_{j,k}+C_{j,k}(x)).
		\end{aligned}\label{tr-eq2}
	\end{equation*}
	Thus we get a system of differential equations with $n^2$ unknown functions 
	$\{C_{i,k}(x);1\leq i,k\leq n\}$	
		\begin{equation} 
		\begin{aligned} 
		 x^{1+\gamma}C_{i,k}'(x) = &
		(\lambda_{i}(x)-\lambda_{k}(x))C_{i,k}(x)\\  & 
		+	\sum_{j=1}^{n}a_{i,j}(x)C_{j,k}(x)+a_{i,k}(x).
		\end{aligned}  \label{Eq-C}
	\end{equation} 
	   The aim is to show that 
		 $\{C_{i,k}(x)\}_{1\leq i\leq n}$ exist in some sectorial region and are
		$\gamma$-Borel summable functions.	
		We construct $C_{i,k}(x)$ by $\gamma$-Laplace integral
	\begin{equation}
		C_{i,k}(x)=\int_{0}^{\infty e^{i\theta}}
		 e^{-(\frac{\xi}{x})^{\gamma}}\widehat{C}_{i,k}(\xi)d\xi^{\gamma}.
	\end{equation}
		Set $\lambda_{i,k}=\lambda_i-\lambda_k$ and 
		$\lambda^*_{i,k}(x)=\lambda_i(x)-\lambda_k(x)-\lambda_{i,k}$. We have from \eqref{Eq-C}
	\begin{equation} 
		\begin{aligned} \ &
		 x^{1+\gamma}C_{i,k}'(x)- \lambda_{i,k}C_{i,k}(x)  = 
		\lambda^*_{i,k}(x)C_{i,k}(x)+ 
		\sum_{j=1}^{n}a_{i,j}(x)C_{j,k}(x)+a_{i,k}(x)
		\end{aligned}  \label{eq-1'}
	\end{equation}	
		and the following system of convolution equations		
	\begin{equation} 
		\begin{aligned} 
		 (\gamma \xi^{\gamma}- \lambda_{i,k})\widehat{C}_{i,k}(\xi) =
		\widehat{\lambda}_{i,k}^{*}(\xi)\underset{\gamma}{*}\widehat{C}_{i,k}(\xi)+ 
		\sum_{j=1}^{n}\widehat{a}_{i,j}(\xi)\underset{\gamma}{*}
		\widehat{C}_{j,k}(\xi)+\widehat{a}_{i,k}(\xi).
		\end{aligned}  \label{eq-conv}
		\end{equation}
	\begin{lem} 
			\begin{enumerate}
				\item[{\rm (1)}] There exists $R>0$ such that 
					$\xi^{\gamma-1}\widehat{\lambda}_{i,k}^{*}(\xi)$ and 
					$\xi^{\gamma-1}\widehat{a}_{i,j}(\xi)$ are holomorphic in 
					$\Xi=\{|\xi|<R\}\cup S(\widehat {I})$.
				\item[\rm{(2)}]
				For arbitrary small $\epsilon>0$
				there exist constants $M_{\epsilon}$ and $c_{\epsilon}$ such that
				\begin{equation}
					\begin{aligned}
						|\widehat{\lambda}_{i,k}^{*}(\xi)|, |\widehat{a}_{i,j}(\xi)| \leq \frac{
						M_{\epsilon}|\xi|^{1-\gamma}e^{c_{\epsilon}|\xi|^{\gamma}}}{\Gamma(1/\gamma)} ,\;
						|\widehat{a}_{i,j}(\xi)|\leq \frac{ M_{\epsilon}|\xi|e^{c_{\epsilon}|\xi|^{\gamma}}}
						{\Gamma((1+\gamma)/\gamma)}   
					\end{aligned}	\label{est-0'}
		  \end{equation}
		  in $\{0<|\xi|<R\}\cup S(\widehat {I}_{\epsilon})$.
			\end{enumerate}	\label{lem-conv}
		\end{lem}
		\begin{proof} The statement (1) is obvious and the estimate \eqref{est-0'} follows from
				${\lambda}_{i,k}^*(x)=O(x)$ and ${a}_{i,j}(x)=O(x^{1+\gamma})$.
		\end{proof}	
			Let us construct $\widehat{C}_{i,k}(\xi)=\sum_{m=1}^{\infty}\widehat{C}_{i,k}^m(\xi)$
			as follows:
			 	\begin{equation} 
			\begin{aligned} 
				(\gamma \xi^{\gamma}- \lambda_{i,k})\widehat{C}_{i,k}^1(\xi)=&
				\widehat{a}_{i,k}(\xi),  \\
				 (\gamma \xi^{\gamma}- \lambda_{i,k})\widehat{C}_{i,k}^m(\xi) = &
				\widehat{\lambda}_{i,k}^*(\xi)\underset{\gamma}{*}\widehat{C}_{i,k}^{m-1}(\xi)+ 
				\sum_{j=1}^{n}\widehat{a}_{i,j}(\xi)\underset{\gamma}{*}
				\widehat{C}_{j,k}^{m-1}(\xi)\;\; m\geq 2.
		\end{aligned}  \label{eq-2'}
	\end{equation}
		Let $\Xi_{\widehat {I}}=\{0< |\xi|<R\}\cup S(\widehat {I})$\; 
		$(|\gamma R^{\gamma}|<\min_{\{i\not = k\}}|\lambda_{i,k}|)$. 
		If $i\not=k$ and $\xi \in \Xi_{\widehat {I}}$, then  
		$\gamma \xi^{\gamma}-\lambda_{i,k}\not =0$ and 
  	$\widehat{C}_{i,k}^1(\xi)=\widehat{a}_{i,k}(\xi)/(\gamma\xi^{\gamma}-\lambda_{i,k})$.
		If $i=k$, then $\widehat{\lambda}_{i,k}^*(\xi)=0$ and
		$\widehat{C}_{k,k}^1(\xi)=\widehat{a}_{k,k}(\xi)/\gamma\xi^{\gamma}$.
		The following lemma holds.	
	\begin{lem} 
		There exist $\widehat{C}_{i,k}^m(\xi)\in {\mathscr O}(S(\widehat{I}))$ such that
	\begin{enumerate} 
		 \item[{\rm (1)\rm}] Let $\epsilon>0$ be an arbitrary small constant. Then there exist constants 
			$A_{\epsilon}$ and $c_{\epsilon}$  such that
		\begin{equation}
			\begin{aligned}
			|\widehat{C}_{i,k}^m(\xi)|
			\leq \frac{A_{\epsilon}^m
			|\xi|^{m-\gamma}}{\Gamma(m/\gamma)}e^{c_{\epsilon}|\xi|^{\gamma}}.\;\; 
			\xi \in \Xi_{\widehat {I}_{\epsilon}}=\{0< |\xi|<R\}\cup S(\widehat {I}_{\epsilon}).
			\end{aligned} \label{est-1}
		\end{equation} 
			\item[{\rm (2)\rm}] 
			$\xi^{\gamma-1}\widehat{C}_{i,k}^m(\xi)\in 
			{\mathscr O}(\{|\xi|<R \})$.
	\end{enumerate} 
	\end{lem}	
	\begin{proof}  
		We show \eqref{est-1} by induction. Let $m=1$.
		For $i\not=k$ inequality \eqref{est-1} holds. 
		Let $i=k$. \eqref{est-1} holds from the second estimate of \eqref{est-0'},
					 Assume \eqref{est-1} holds for $m-1$. 
			 Then there exists a constants $B_{\epsilon}$ such that 
	\begin{equation*} 
		\begin{aligned} \ &
		|\widehat{\lambda}_{i,k}^*(\xi)\underset{\gamma}{*}\widehat{C}_{i,k}^{m-1}(\xi)|, 
		|\widehat{a}_{i,j}(\xi)\underset{\gamma}{*}
		\widehat{C}_{j,k}^{m-1}(\xi)|\leq \frac{B_{\epsilon}A^{m-1}_{\epsilon}
		|\xi|^{m-\gamma}e^{c_{\epsilon}|\xi|^{\gamma}}}{\Gamma(m/\gamma)}. \\
		\end{aligned} 
	\end{equation*}	
		If $i\not=k$, $\gamma \xi^{\gamma}-{\lambda}_{i,k}\not =0$. Then estimate \eqref{est-1} for $m$ holds.
		If $i=k$, it follows from $\widehat{\lambda}_{i,k}^*(\xi)=0$ and 
		the second estimate of \eqref{est-0'} that
	\begin{equation*}
		\begin{aligned}
		|\widehat{a}_{k,j}(\xi)\underset{\gamma}{*}\widehat{C}_{j,k}^{m-1}(\xi)|
		\leq \frac{B_{\epsilon}A^{m-1}_{\epsilon}|\xi|^{m}e^{c_{\epsilon}|\xi|^{\gamma}}}{\Gamma(m/\gamma)}. 
		\end{aligned} 
	\end{equation*}
	ans the estimate \eqref{est-1} also holds for $m$. We have
	the statement (2) in the same way as above.  
	\end{proof}
	Thus we get
		\begin{prop}There exist $\{\widehat{C}_{i,k}(\xi)\}_{1\leq i,k \leq n}$ such that
		\begin{enumerate} 
			\item[{\rm (1)\rm}] $\widehat{C}_{i,k}(\xi)\in 
			{\mathscr O}(\Xi_{\widehat {I}})$, 
			$\xi^{\gamma-1}\widehat{C}_{i,k}(\xi)\in 
			{\mathscr O}(\{|\xi|<R\})$
			and $\{\widehat{C}_{i,k}(\xi)\}_{1\leq i ,k \leq n}$ satisfy convolution equations 
			\eqref{eq-conv}. 
			\item[{\rm (2)\rm}] For any small $\epsilon>0$ there are positive constants 
			$M_{\epsilon}$ and $c_{\epsilon}'$ such that
		\begin{equation}
			\begin{aligned}
			|\widehat{C}_{i,k}(\xi)|
			\leq \frac{M_{\epsilon}
			|\xi|^{1-\gamma}}{\Gamma(1/\gamma)}e^{c_{\epsilon}'|\xi|^{\gamma}}, \;
				\xi \in \Xi_{\widehat{I}_{\epsilon}}=\{\xi;0< |\xi|<R\}\cup S(\widehat {I}_{\epsilon}).
			\end{aligned}
		\end{equation} 	
			\end{enumerate}  
	\end{prop}	
		\begin{proof} We have 
	$\{\widehat{C}_{i,k}^m(\xi)\}_{m=1,2,	\cdots} $ with \eqref{est-1}. Then 
	there exist constants $M_{\epsilon},c_{\epsilon}$ and $c_{\epsilon}'$ such that
	\begin{equation*}
		\begin{aligned}
		\sum_{m=1}^{\infty}
		|\widehat{C}_{i,k}^m(\xi)|
		\leq \sum_{m=1}^{\infty}\frac{A^{m}_{\epsilon}|\xi|^{m-\gamma}}{\Gamma(m/\gamma)}
		e^{c_{\epsilon}|\xi|^{\gamma}}\leq M_{\epsilon}|\xi|^{1-\gamma} e^{c_{\epsilon}'\xi|^{\gamma}}.
		\end{aligned} 
	\end{equation*}
	Hence $\widehat{C}_{i,k}(\xi)=\sum_{m=1}^{\infty}\widehat{C}_{i,k}^m(\xi)\;
	(1\leq i,k \leq n)$ converge and satisfy \eqref{eq-conv}. 
	\end{proof}	
	  Let us define for $\theta \in I_{\epsilon}=(\theta_*-\epsilon, \theta_*+\epsilon)$ 
	\begin{equation}
		\begin{aligned}
		C_{i,k}(x)=\int_{0}^{e^{\i \theta}\infty} e^{-(\frac{\xi}{x})^{\gamma}}
		\widehat{C}_{i,k}(\xi)d\xi^{\gamma}.
		\end{aligned}
	\end{equation}
		Thus we obtain a linear transformation $Y=(I+C(x))Z$  and Proposition \ref{prop-1} is shown. 
	
\end{document}